\documentclass{amsart}

\usepackage{amsmath,latexsym,amssymb,amsthm,graphicx}
\usepackage{hyperref}
\usepackage[all]{xy}
\usepackage{enumitem}
\setenumerate[1]{label=(\thesection.\arabic*)}

\numberwithin{equation}{section}

\newtheorem{theorem}{Theorem}[section]
\newtheorem{lemma}[theorem]{Lemma}

\newtheorem{corollary}[theorem]{Corollary}
\newtheorem{proposition}[theorem]{Proposition}

\theoremstyle{definition}
\newtheorem{remark}[theorem]{Remark}
\newtheorem{remarks}[theorem]{Remarks}
\newtheorem{example}[theorem]{Example}

\def\N{\ensuremath{\mathbb{N}}}
\def\Z{\ensuremath{\mathbb{Z}}}

\def\R{\ensuremath{\mathbb{R}}}
\def\C{\ensuremath{\mathbb{C}}}

\newcommand{\pa}[1]{\left(#1\right)}
\newcommand{\cpa}[1]{\left\{#1\right\}}
\newcommand{\wt}[1]{\widetilde{#1}}
\newcommand{\tn}[1]{\textnormal{#1}}
\newcommand{\br}[1]{\left[#1\right]}

\newcommand{\cov}[1]{\tn{Cov}\pa{#1}}
\newcommand{\h}[1]{\tn{Hom}\pa{#1}}
\def\fa{f^{\ast}}
\def\fs{f_{\sharp}}
\newcommand{\bs}[1]{\pa{\uppercase{#1},{\lowercase{#1}}_{0}}}

\newcommand{\im}[1]{\tn{Im} \, #1}
\newcommand{\pr}[1]{\tn{pr}_{#1}}
\newcommand{\rest}[1]{\left.#1\right|}
\newcommand{\card}[1]{\left|#1\right|}

\def\cov{\tn{Cov}}
\def\bcov{\tn{BCov}}
\def\scov{\tn{SCov}}
\def\bscov{\tn{BSCov}}

\font\cuf=cmtt8
\newcommand{\curl}[1]{{\cuf #1}}

\begin{document}
\title{Topological pullback, covering spaces, and a triad of Quillen}

\author[J.S.~Calcut]{Jack S. Calcut}
\address{Department of Mathematics\\
         Oberlin College\\
         Oberlin, OH 44074}
\email{jcalcut@oberlin.edu}
\urladdr{\href{http://www.oberlin.edu/faculty/jcalcut/}{\curl{http://www.oberlin.edu/faculty/jcalcut/}}}

\author[J.D.~McCarthy]{John D. McCarthy}
\address{Department of Mathematics\\
                    Michigan State University\\
                    East Lansing, MI 48824-1027}
\email{mccarthy@math.msu.edu}
\urladdr{\href{http://www.math.msu.edu/~mccarthy/}{\curl{http://www.math.msu.edu/\textasciitilde mccarthy/}}}

\date{May 14, 2012}

\begin{abstract}
We study pullback from a topological viewpoint with emphasis on pullback of covering maps.
We generalize a triad of Quillen on properties of the pullback functor.
\end{abstract}

\maketitle

\section{Introduction}\label{s:intro}

The \emph{pullback} operation---also called \emph{fiber product}---is useful in various settings including
vector bundles~\cite[pp.~97,~171]{hirsch}, fiber bundles~\cite[pp.~47--48]{steenrod}, schemes~\cite[pp.~87--90]{hartshorne}, and categories~\cite[pp.~44,~146,~228]{szamuely}.
We give a topological introduction to pullback with emphasis on pullback of covering maps.
As an application, we generalize a triad, observed by Quillen~\cite[pp.~114--116]{quillen}, of properties of the pullback functor.\\

Let $f:P\to\cpa{v}$ be a map of posets where $\cpa{v}$ is a singleton.
Let $\cov\pa{P}$ be the category of local systems on $P$.
Let \hbox{$\fa:\cov\pa{\cpa{v}}\to\cov\pa{P}$} denote the pullback functor. Quillen observed the following:

\begin{enumerate}\setcounter{enumi}{\value{equation}}
\item\label{q1} $P$ is $\pa{-1}$-connected (i.e., nonempty) if and only if $\fa$ is faithful.
\item\label{q2} $P$ is $0$-connected (i.e., nonempty and connected) if and only if $\fa$ is full and faithful.
\item\label{q3} $P$ is $1$-connected (i.e., $0$-connected and simply-connected) if and only if $\fa$ is an equivalence of categories.
\setcounter{equation}{\value{enumi}}
\end{enumerate}

For our generalization of Quillen's triad, let $f:X\to Y$ be a \textbf{map} (= continuous function) of topological spaces.
Let $\cov\pa{X}$ be the category of coverings of $X$.
Let $\fa:\cov\pa{Y}\to\cov\pa{X}$ denote the pullback functor.
Let $\Gamma\pa{X}$ denote the set of connected components of $X$.
Then:

\begin{enumerate}\setcounter{enumi}{\value{equation}}
\item\label{Q1} $\fs:\Gamma\pa{X}\to\Gamma\pa{Y}$ is surjective if and only if $\fa$ is faithful.
\item\label{Q2} $\fs:\pi_0\pa{X}\to\pi_0\pa{Y}$ is a bijection and $\fs:\pi_1\pa{X,x}\to\pi_1\pa{Y,f\pa{x}}$ is surjective for each $x\in X$ if and only if $\fa$ is full and faithful.
\item\label{Q3} $\fs:\pi_0\pa{X}\to\pi_0\pa{Y}$ is a bijection and $\fs:\pi_1\pa{X,x}\to\pi_1\pa{Y,f\pa{x}}$ is an isomorphism for each $x\in X$ if and only if $\fa$ is an equivalence of categories.
\setcounter{equation}{\value{enumi}}
\end{enumerate}

Equivalences~\ref{Q1}--\ref{Q3} generalize Quillen's triad in three ways:
(1) the target $Y$ is not required to be a singleton,
(2) spaces are much more general than posets or simplicial complexes,
and (3) we prove all three equivalences for four different categories of coverings.
Note that Quillen observed the poset analogue of~\ref{Q3} in~\cite[p.~116]{quillen}.\\

The equivalence~\ref{Q1} is proved in Proposition~\ref{gq1} and uses the hypothesis: $Y$ is locally connected.
If $Y$ is locally path-connected, then a convenient alternative generalization of~\ref{q1} is given in Corollary~\ref{gq1_pc}.
The equivalence~\ref{Q2} is proved in Proposition~\ref{gq2} and uses the hypotheses: $X$ and $Y$ are locally path-connected and $Y$ is semilocally simply-connected.
The equivalence~\ref{Q3} is proved in Proposition~\ref{gq3} and uses the additional hypothesis: $X$ is semilocally simply-connected.
We give examples to show that some of these key hypotheses may not be omitted.
For instance, Example~\ref{ha_ex} shows that the backward implication in~\ref{Q3} is false when $X$ is not semilocally simply-connected.
In this example, $X$ is the \emph{Harmonic archipelago}, an interesting space introduced by Bogley and Sieradski~\hbox{\cite[pp.~6--7]{bogley_sieradski}}
(see Figure~\ref{harm_arch} below).\\

Equivalences~\ref{Q1}--\ref{Q3} do not require spaces to be connected, locally simply-connected, nor even Hausdorff.
Furthermore, all three equivalences are proved for four categories of coverings, namely $\cov$, $\scov$, $\bcov$, and $\bscov$ (see Section~\ref{ss:catcov} for details).
The category $\cov\pa{X}$ has as objects all coverings of $X$, where empty fibers are permitted.
The category $\scov\pa{X}$ has as objects all surjective coverings of $X$.
The categories $\bcov\pa{X,x_0}$ and $\bscov\pa{X,x_0}$ are the based versions of $\cov\pa{X}$ and $\scov\pa{X}$ respectively.\\

When considering all coverings of a space $X$, it is natural to permit empty fibers for two reasons.
First, if $X$ is path-connected, locally path-connected, and semilocally simply-connected, then $\cov\pa{X}$ is equivalent to the category of $G$-sets where $G:=\pi_1\bs{X}$.
And, the empty set is a $G$-set.
Second, if $X$ is not connected, then the natural definition of covering map permits cardinalities of fibers to vary, so forbidding cardinality zero seems artificial.
Hatcher~\cite[p.~56]{hatcher} also adopts the convention that fibers may be empty, and does not require spaces to be Hausdorff.
We recommend~\cite[Ch.~1]{hatcher} as the ideal prerequisite to the present paper.
Spanier~\cite[Ch.~2]{spanier} is a classic and useful reference.
We also recommend M{\o}ller's notes~\cite{moller} and Bar-Natan's short, stimulating note~\cite{barnatan}.
We are not aware of a comprehensive, topological introduction to pullback in the literature.
One purpose of the present paper is to provide such an introduction.\\

In a follow up paper, we extend Quillen's triad in several new directions.
Let $f:X\to Y$ be a map of reasonably nice topological spaces.
Let $\fa:\cov\pa{Y}\to\cov\pa{X}$ denote the pullback functor.
We say $\fa$ has \textbf{nullity-zero} provided: if $\fa\pa{E}$ is trivial, then $E$ is trivial (trivial covers are defined in sections~\ref{ss:coverings} and~\ref{ss:catcov} below).
We give algebraic equivalents for $\fa$ to be essentially injective, to be essentially surjective, and to have nullity-zero.
Concerning essential injectivity, we prove a \emph{Tannakian-like} result~\cite{joyal_street}.
Namely, failure of $\fa$ to be essentially injective may be detected using only \emph{finite} component covers of $Y$ (but generally \emph{not} of $X$).
This holds for arbitrary fundamental groups of $X$ and $Y$, even infinite.
The finite fundamental group case utilizes Burnside rings and raises several open questions.\\

This paper is organized as follows.
Section~\ref{ss:pullback} defines pullback and presents some fundamental properties.
Sections~\ref{ss:coverings} and~\ref{ss:pbcover} discuss covering maps and pullback.
Section~\ref{ss:catcov} presents four categories of coverings.
Sections~\ref{ss:ducov} and~\ref{ss:pbducov} discuss disjoint union and pullback, including examples to show where care is necessary.
Section~\ref{s:gqt} uses the material from Section~\ref{s:pullback} to generalize Quillen's triad.\\

We close this introduction by recalling connections between posets and simplicial complexes.
This material is not used below, it merely explains how Quillen's triad is translated to a topological setting.
To each poset $P$ one associates the \emph{order complex} $\card{P}$ of $P$, namely the simplicial complex with vertex set $P$ and a simplex for each nonempty, finite chain (= totally ordered subset) in $P$.
This association permits topological properties to be attributed to posets~\cite[p.~103]{quillen}.
To each simplicial complex $K$ one associates the \emph{face poset} $S(K)$ of $K$, namely the poset of nonempty simplices in $K$ partially ordered by inclusion.
Although simple examples show not every simplicial complex arises as an order complex (three vertices suffice) and vice versa, the two associations are intimately related: $\card{S(K)}$ is the barycentric subdivision of $K$.
In addition, there is an equivalence between the categories $\cov\pa{P}$ and $\cov(\card{P})$ (see~\cite[\S 1,7]{quillen} and~\cite[App.~I]{gabriel_zisman}).
Therefore, Quillen's triad is equivalent to the analogous results for a simplicial map $f:K\to\cpa{v}$, and~\ref{Q1}--\ref{Q3} indeed generalize~\ref{q1}--\ref{q3}.

\section{Pullback, Coverings, and Disjoint Union}\label{s:pullback}

\subsection{Pullback}\label{ss:pullback}

Let a diagram of maps of topological spaces be given:
\begin{equation}\label{givens}\begin{split}
\xymatrix{
    							&	Z	\ar[d]^-{g}\\
    X	\ar[r]^{f}	&	Y}
\end{split}\end{equation}
Consider the following diagram (noncommutative in general):
\begin{equation}\label{noncomm}\begin{split}
\xymatrix{
    X\times Z	\ar[r]^-{\pr{2}}	\ar[d]_{\pr{1}}	&	Z	\ar[d]^{g}\\
    X  							\ar[r]^-{f}     																& Y }
\end{split}\end{equation}
where $\pr{1}$ and $\pr{2}$ are the coordinate projections.
The \textbf{pullback} of $g$ along $f$ consists of the subspace:
\begin{equation}\label{pullback_set}
	\fa\pa{Z} := \cpa{ \pa{x,z}\in X\times Z \mid f(x)=g(z) }\subset X\times Z
\end{equation}
and the commutative diagram:
\begin{equation}\label{pullback}\begin{split}
\xymatrix{
    \fa\pa{Z}	\ar[r]^-{\wt{f}}	\ar[d]_{\fa(g)}	&	Z	\ar[d]^{g}\\
    X  							\ar[r]^-{f}     																& Y }
\end{split}\end{equation}
Here, $\fa(g)$ and $\wt{f}$ are continuous, being the restrictions to $\fa\pa{Z}$ of $\pr{1}$ and $\pr{2}$.
Note that $\fa\pa{Z}$ is the equalizer\footnote{The \textbf{equalizer} of two functions $h,k:A\to B$ is the set of $a\in A$ such that $h(a)=k(a)$.} of the maps $f\circ\pr{1}$ and $g\circ\pr{2}$, and is the largest subset of $X\times Z$ on which~\eqref{noncomm} commutes.\\

If~\eqref{givens} happens to be based, say with $f(x_0)=g(z_0)$, then~\eqref{pullback} is naturally based with $(x_0,z_0)\in\fa\pa{Z}$.
If $x\in X$ and $y:=f(x)$, then in~\eqref{pullback} we have:
\begin{enumerate}\setcounter{enumi}{\value{equation}}
\item\label{vert_fiber} The fiber over $x$ equals $\cpa{x}\times g^{-1}\pa{y}$ (possibly empty).
\item\label{bij_vert_fibers} The map $\wt{f}$ restricts to a homeomorphism of fibers $\cpa{x}\times g^{-1}\pa{y}\to g^{-1}(y)$.
\setcounter{equation}{\value{enumi}}
\end{enumerate}

The pullback~\eqref{pullback} satisfies a well known and easily verified universal property. Namely, if $Q$, $q_1$, and $q_2$ are given so that the following diagram commutes:
\begin{equation}\label{uppb}\begin{split}
\xymatrix{
	Q	\ar@/^/[rrd]^{q_2}	\ar@{-->}[rd]_{\mu}	\ar@/_/[rdd]_{q_1} &	&\\
		&	\fa\pa{Z} \ar[d]	\ar[r]	&	Z	\ar[d]\\
		&	X	\ar[r]	&	Y}
\end{split}\end{equation}
then there exists a unique map $\mu:Q\to\fa\pa{Z}$ making the entire diagram~\eqref{uppb} commute. Evidently, $\mu(c)=\pa{q_1(c),q_2(c)}$.

\begin{example}
Let $Y=S^1\subset\C$, let $X=\cpa{1}\subset S^1$, and let $f:X\to Y$ be inclusion.
If $g:\R\to Y$ is the universal covering $t\mapsto\exp\pa{2\pi it}$, then $\fa\pa{\R}$ is a copy of $\Z$.
Hence, pullback yields disconnected coverings straightaway.
\end{example}

\begin{remark}
There is an obvious symmetry in the definition of pullback.
One could just as well pullback $f$ along $g$, and $g^{\ast}\pa{X}$ is canonically homeomorphic to $\fa\pa{Z}$ by the map $\pa{z,x}\mapsto\pa{x,z}$.
For this reason, $\fa\pa{Z}$ is sometimes denoted $X\times_{Y} Z$ in the literature and is sometimes called the fiber product of $X$ and $Z$ over $Y$.
We are mainly interested in pulling back arbitrary coverings of $Y$ along a fixed map $f:X\to Y$, so we stick to the somewhat asymmetric $\fa$ notation.
Still, the aforementioned symmetry is useful.
For instance, by symmetry, properties~\ref{vert_fiber} and~\ref{bij_vert_fibers} have obvious analogues for horizontal fibers.
It follows immediately that if $f$ is injective, then $\wt{f}$ is injective.
And, again by symmetry, if $g$ is injective, then $\fa\pa{g}$ is injective.
\end{remark}

\begin{lemma}\label{inclusion}
Let a pullback diagram~\eqref{pullback} be given.
Then, $\im{\wt{f}}=g^{-1}\pa{\im{f}}$ and $\im{\fa\pa{g}}=f^{-1}\pa{\im{g}}$.
If $f$ is inclusion, then $\wt{f}$ is an embedding.
If $f$ is a homeomorphism, then $\wt{f}$ is a homeomorphism.
\end{lemma}

\begin{proof}[Proof of Lemma~\ref{inclusion}]
If $f$ is inclusion, then restricting the codomain of $\wt{f}$ to $\im{\wt{f}}$ yields the bijective map $\fa\pa{Z}\to \im{\wt{f}}$, whose inverse map is $z\mapsto\pa{g\pa{z},z}$.
If $f$ is a homeomorphism, then $\wt{f}$ is a bijective map with inverse map $z\mapsto\pa{f^{-1} g\pa{z},z}$.
\end{proof}

The next lemma says that the pullback of a pullback is naturally homeomorphic to the pullback along the composition (cf.~\cite[p.~49]{steenrod}).

\begin{lemma}\label{pullback_comp}
If a diagram of maps is given:
\begin{equation}\label{comp_given}\begin{split}
\xymatrix{
    							&								&	Z	\ar[d]^-{g}\\
    W	\ar[r]^{h}	&	X	\ar[r]^{f}	&	Y}
\end{split}\end{equation}
then the map $\mu:h^{\ast}\pa{\fa\pa{Z}}\to \pa{fh}^{\ast}\pa{Z}$
given by $\pa{w,\pa{x,z}}\mapsto\pa{w,z}$ is a homeomorphism,
and the following diagram commutes:
\begin{equation}\label{comp_conclusion}\begin{split}
\xymatrix{
	\pa{fh}^{\ast}\pa{Z}	\ar@/^/[rrrd]^{\wt{fh}}	\ar@/_/[rdd]_{\pa{fh}^{\ast}\pa{g}} &	&\\
		&	h^{\ast}\pa{\fa\pa{Z}}	\ar[r]	\ar[d]^{h^{\ast}\pa{\fa\pa{g}}}	\ar[lu]^{\mu}	&	\fa\pa{Z} \ar[d]^{\fa\pa{g}}	\ar[r]	&	Z	\ar[d]^{g}\\
		&	W	\ar[r]^{h}	&	X	\ar[r]^{f}	&	Y}
\end{split}\end{equation}
\end{lemma}

\begin{proof}[Proof of Lemma~\ref{pullback_comp}]
Apply the universal property of pullback to get $\mu$, then note that the inverse map of $\mu$ is $\pa{w,z}\mapsto\pa{w,\pa{h\pa{w},z}}$.
\end{proof}

Any embedding factors as a homeomorphism followed by an inclusion, so Lemmas~\ref{pullback_comp} and~\ref{inclusion} immediately yield the following.

\begin{corollary}\label{embedding}
If $f$ is an embedding in~\eqref{pullback}, then $\wt{f}$ is an embedding.
\end{corollary}

\subsection{Coverings}\label{ss:coverings}

A \textbf{fiber bundle projection} is a map $p:E\to Y$ satisfying: for each $y\in Y$ there exists an open neighborhood $U$ of $y$ in $Y$, a space $F$, and a homeomorphism $\varphi:p^{-1}\pa{U}\to	U\times F$ such that the following diagram commutes:
\begin{equation}\begin{split}\label{localtrivialization}
\xymatrix{
    p^{-1}\pa{U}	\ar[r]^-{\varphi}	\ar[d]_{\rest{p}}	&	U\times F	\ar[dl]^{\pr{1}}\\
    U }
\end{split}\end{equation}
The data $\pa{U,F,\varphi}$ is a \textbf{local trivialization} of $p$ at $y\in Y$, $U$ is an \textbf{evenly covered neighborhood} of $y$ in $Y$, and $F$ is a \textbf{fiber}.
We allow fibers to be empty.
If $y\in Y$ and $\pa{U,F,\varphi}$ is any local trivialization of $p$ at $y$, then the fiber $p^{-1}\pa{y}\subset E$ is homeomorphic to $F$.
For any fixed space $F$, the set of points in $Y$ with fiber homeomorphic to $F$ is open and closed in $Y$.
Hence:
\begin{enumerate}\setcounter{enumi}{\value{equation}}
\item\label{fibers_homeo} Over each component of $Y$, fibers of $p$ are homeomorphic.
\setcounter{equation}{\value{enumi}}
\end{enumerate}

A \textbf{covering map} is a fiber bundle projection with all fibers discrete.
If \hbox{$p:E\to Y$} is a covering map and $\pa{U,F,\varphi}$ is a local trivialization of $p$ at some point $y\in Y$, then for each $d\in F$ the set $V:=\varphi^{-1}\pa{U\times\cpa{d}}$ is open in $E$ and $\rest{p}V:V\to U$ is a homeomorphism.
It follows that each covering map is a local homeomorphism and is open.
Covering maps need not be closed: consider $\R\to S^1$, given by $t\mapsto\exp\pa{2\pi it}$, and $\cpa{n+\frac{1}{n+1}\mid n\in\N}\subset\R$.\\

Given fiber bundle projections (or covering maps) $p_1:E_{1}\to Y$ and $p_2:E_2\to Y$, a \textbf{morphism} $t:p_1\to p_2$ is a map $t:E_1\to E_2$ such that the following diagram commutes:
\begin{equation}\begin{split}\label{morphismofcovers}
\xymatrix{
    E_{1}	\ar[rr]^-{t}	\ar[dr]_{p_{1}}	&	&	E_{2}	\ar[dl]^{p_{2}}\\
    &	Y }
\end{split}\end{equation}
A morphism $t$ as in~\eqref{morphismofcovers} is an \textbf{isomorphism} provided there exists a morphism $s:p_2\to p_1$ such that
$s\circ t=\tn{id}_{E_{1}}$ and $t\circ s=\tn{id}_{E_{2}}$.
Plainly, a morphism $t$ as in~\eqref{morphismofcovers} is an isomorphism if and only if $t:E_1\to E_2$ is a homeomorphism.
Given covering maps $p_1:E_{1}\to Y$ and $p_2:E_2\to Y$,
we write $p_1\cong p_2$ or $E_1\cong E_2$ to mean there exists an isomorphism $t:p_1\to p_2$.\\

A covering map $p:E\to Y$ is \textbf{trivial} provided there exists a discrete space $D$ (possibly empty) and an isomorphism
$t:p\to\pr{1}$ where $\pr{1}:Y\times D \to Y$.
If $Y$ is nonempty, then $Y$ has infinitely many isomorphism classes of trivial covers, one for each cardinal number.

\begin{remarks}
Covering maps, and morphisms between them, may well fail to be surjective.
The empty covering of $Y\neq\emptyset$ is not surjective.
If $Y$ is disconnected (and locally connected, say), then fibers of a single cover of $Y$ may have varying cardinalities (zero included).
For morphisms, let $Y=S^1$ and consider the obvious trivial covers $E_1=S^1\times\cpa{1}$ and $E_2=S^1\times\cpa{1,2}$. Then, inclusion $E_1\to E_2$ is a morphism, but is not surjective.
\end{remarks}

\begin{lemma}\label{restcov}
Let $p:E\to Y$ be a covering map.
Let $A\subset Y$.
Then, the restriction $\rest{p}:p^{-1}\pa{A}\to A$ is a covering map.
\end{lemma}

\begin{proof}[Proof of Lemma~\ref{restcov}]
Let $y\in A$.
Let $\pa{U,F,\varphi}$ be a local triviliazation of $p$ at $y\in Y$.
Let $V:=U\cap A$.
Then, $\pa{V,F,\rest{\varphi}p^{-1}\pa{V}}$ is a local trivialization of $\rest{p}$ at $y$.
\end{proof}

\begin{lemma}\label{morphismiscover}
Let $t:p_1\to p_2$ be a morphism of covering maps as in~\eqref{morphismofcovers}. If $Y$ is locally connected, then $t$ is a covering map. 
\end{lemma}

\begin{proof}[Proof of Lemma~\ref{morphismiscover}]
Fix $e' \in E_2$ and let $y':=p_2\pa{e'}$.
Let $\pa{U_i,F_i,\varphi_i}$ be a local trivialization of $p_i$ at $y'$ for each $i\in\cpa{1,2}$.
Let $U$ be the connected component of $U_1\cap U_2$ containing $y'$, which is open in $Y$ since $Y$ is locally connected.
By restriction, we have local trivializations $\pa{U,F_i,\psi_i}$ of $p_i$ at $y'$ for each $i\in\cpa{1,2}$.
The following diagram commutes:
\begin{equation}\label{morphcoverdiag}\begin{split}
\xymatrix{
    U\times F_1	\ar[drr]_{\pr{1}}	&	\ar[l]_{\psi_{1}}	p_{1}^{-1}\pa{U}	\ar[rr]^-{\rest{t}}	\ar[dr]^-{\rest{p_{1}}}	&	&
    			p_{2}^{-1}\pa{U}	\ar[dl]_-{\rest{p_{2}}}	\ar[r]^{\psi_{2}}	&	U\times F_2	\ar[dll]^{\pr{1}}\\
    &	&	U }
\end{split}\end{equation}
Define $s:U\times F_1 \to U\times F_2$ by $s:=\psi_{2}\circ \rest{t}\circ \psi_{1}^{-1}$.
As $U$ is connected, $F_1$ and $F_2$ are discrete, and~\eqref{morphcoverdiag} commutes, we see that $s\pa{y,d}=\pa{y,\sigma\pa{d}}$ for some map \hbox{$\sigma:F_1\to F_2$}.
Let $d'\in F_2$ be the unique element of the fiber such that $\psi_2\pa{e'}=\pa{y',d'}$.
Let $V:=\psi_2^{-1}\pa{U\times\cpa{d'}}$.
It is straightforward to verify that $\pa{V,\sigma^{-1}\pa{d'},\varphi}$ is a local trivialization of $t$ at $e'$, where $\varphi\pa{e}:=\pa{t\pa{e},\pr{2}\circ\psi_1\pa{e}}$.
\end{proof}

The previous lemma becomes false without the local connectivity hypothesis on $Y$, as shown by the following example.

\begin{example}
Let $Y:=\cpa{0}\cup\cpa{1/k \mid k\in\N}\subset\R$ and $\Z^{\ast}:=\Z-\cpa{0}$. Consider the commutative diagram:
\begin{equation}\begin{split}\label{}
\xymatrix{
    Y\times\Z^{\ast}	\ar[rr]^-{t}	\ar[dr]_{\pr{1}}	&	&	Y\times\N	\ar[dl]^{\pr{1}}\\
    &	Y }
\end{split}\end{equation}
where $t$ is defined by:
\begin{equation}\begin{split}
\xymatrix@R=0pt{
	Y\times\Z^{\ast}	\ar[r]^-{t}			&	Y\times\N\\
	\pa{y,n}									\ar@{|-{>}}[r]	&	\pa{y,-n}		&	n\in -\N\\
	\pa{0,n}									\ar@{|-{>}}[r]	&	\pa{0,n}		&	n\in\N\\
	\pa{1/k,n}								\ar@{|-{>}}[r]	&	\pa{1/k,n}	&	k\geq n\in\N\\
	\pa{1/k,n}								\ar@{|-{>}}[r]	&	\pa{1/k,1}	&	1\leq k<n\in\N}
\end{split}\end{equation}
So, $t$ is a (surjective) morphism of trivial covers of $Y$. There is no local trivialization of $t$ at $\pa{0,1}$ by comparing cardinalities of fibers.
\end{example}

\begin{corollary}
A bijective morphism $t$ of covers of a locally connected space $Y$ is an isomorphism.
\end{corollary}

\begin{proof}
By Lemma~\ref{morphismiscover}, $t$ is a covering map. So, $t$ is open.
\end{proof}

\begin{lemma}\label{homeotocover}
If a commutative diagram of maps is given:
\begin{equation}\begin{split}\label{homeotocoverdiag}
\xymatrix{
    Z	\ar[rr]^-{t}	\ar[dr]_{g}	&	&	E	\ar[dl]^{p}\\
    &	Y }
\end{split}\end{equation}
where $p$ is a covering map and $t$ is a homeomorphism, then $g$ is a covering map and $Z\cong E$.
The result also holds with the arrow of $t$ reversed.
\end{lemma}

\begin{proof}[Proof of Lemma~\ref{homeotocover}]
Let $y\in Y$. Let $\pa{U,F,\varphi}$ be a local trivialization of $p$ at $y$.
It is straightforward to check that $t\pa{g^{-1}\pa{U}}=p^{-1}\pa{U}$.
Hence, $\pa{U,F,\varphi\circ\pa{\rest{t}g^{-1}\pa{U}}}$ is a local trivialization of $g$ at $y$.
Therefore, $g$ is a covering map, and $Z\cong E$ since $t$ is a homeomorphism.
\end{proof}

The closed conclusion in the next lemma is immediate in case $E$ is Hausdorff, but we do not assume spaces are Hausdorff.
For a simple proof of the lemma, see~\cite[Prop.~1.34]{hatcher}.

\begin{lemma}\label{equalizer_lift}
Let $p:E\to Y$ be a covering map.
Let $g:W\to Y$ be any map.
Let $t_1$ and $t_2$ be continuous lifts of $g$ to $E$ (i.e., $p\circ t_1=g$ and $p\circ t_2=g$).
Then, the equalizer of $t_1$ and $t_2$ is open and closed in $W$.
\end{lemma}

Lemma~\ref{equalizer_lift} immediately implies the following.

\begin{corollary}\label{equalizer_morphisms}
Let $p_1:E_1\to Y$ and $p_2:E_2\to Y$ be covering maps. Let $t_1$ and $t_2$ be morphisms $p_1\to p_2$.
Then, the equalizer of $t_1$ and $t_2$ is open and closed in $E_1$.
\end{corollary}

\subsection{Pullback of a Cover}\label{ss:pbcover}

The following key lemma pulls back a commutative triangle and will be used for pulling back morphisms.

\begin{lemma}\label{pullbacktriangle}
If a commutative diagram of maps is given:
\begin{equation}\begin{split}
\xymatrix{
    &							&	Z_{1}  \ar[dl]_{g_{1}}	\ar[dd]^{t}\\
    X	\ar[r]^{f}	&	Y\\
    &							&	Z_{2}  \ar[ul]^{g_{2}}}
\end{split}\end{equation}
then pullback yields the commutative diagram:
\begin{equation}\begin{split}\label{pullbackmorphism}
\xymatrix{
    \fa\pa{Z_{1}}	\ar[rrr]^-{\wt{f}}	\ar[rd]^-{\fa\pa{g_1}}	\ar[dd]_{\fa\pa{t}}	&	&	&	Z_{1}  \ar[dl]_{g_{1}}	\ar[dd]^{t}\\
    &	X	\ar[r]^{f}																																								&	Y\\
    \fa\pa{Z_{2}}	\ar[rrr]	\ar[ru]_-{\fa\pa{g_2}}															&	&	&	Z_{2}  \ar[ul]^{g_{2}}}
\end{split}\end{equation}
where $\fa\pa{t}=\rest{\pa{\tn{id}_{X}\times t}}{\fa\pa{Z_{1}}}$.
Furthermore, if $t$ is injective, surjective, open, or a homeomorphism respectively, then $\fa\pa{t}$ is as well.
\end{lemma}

\begin{proof}[Proof of Lemma~\ref{pullbacktriangle}]
Pullback $g_1$ and $g_2$ (separately) along $f$ to obtain~\eqref{pullbackmorphism} minus $\fa\pa{t}$.
To get $\fa\pa{t}$, apply the universal property of pullback to the pullback of $g_2$ along $f$, and with $Q=\fa\pa{Z_1}$, $q_1=\fa\pa{g_1}$, and $q_2=t\circ\wt{f}$.
The injective, surjective, and homeomorphism claims are simple exercises.
Suppose that $t$ is open. It suffices to verify $\fa\pa{t}$ is open on a basis of $\fa\pa{Z_1}$.
If $A\subset X$ and $B\subset Z_1$, then:
\[
	\fa\pa{t}\pa{\pa{A\times B} \cap \fa\pa{Z_1}} = \pa{A\times t\pa{B}}\cap \fa\pa{Z_2}
\]
and the desired result follows since $t$ is open.
\end{proof}

Consider the following diagram where $f$ is a map and $p$ is a covering map:
\begin{equation}\begin{split}\label{basicdiagram}
\xymatrix{
    							&	E	\ar[d]^-{p}\\
    X	\ar[r]^{f}	&	Y}
\end{split}\end{equation}
Pullback yields the commutative diagram:
\begin{equation}\begin{split}\label{basicpullback}
\xymatrix{
    \fa\pa{E}	\ar[r]^-{\wt{f}}	\ar[d]_{\fa(p)}	&	E	\ar[d]^{p}\\
    X  							\ar[r]^-{f}     																& Y }
\end{split}\end{equation}

\begin{lemma}\label{pullbackiscover}
In~\eqref{basicpullback}, $\fa(p)$ is a covering map.
\end{lemma}

\begin{proof}[Proof of Lemma~\ref{pullbackiscover}]
Let $x\in X$, let $y:=f\pa{x}$, and let $\pa{U,F,\varphi}$ be a local trivialization of $p$ at $y\in Y$.
By Lemma~\ref{pullbacktriangle}, we may pullback diagram~\eqref{localtrivialization} along the map $\rest{f}:f^{-1}\pa{U}\to U$ and obtain the commutative diagram:
\begin{equation}\begin{split}
\xymatrix{
    \rest{f}^{\ast}\pa{p^{-1}\pa{U}}	\ar[rrr]	\ar[rd]^{\rest{f}^{\ast}(\rest{p})}	\ar[dd]_{\rest{f}^{\ast}\pa{\varphi}}	&	&	&	p^{-1}\pa{U}  \ar[dl]_{\rest{p}}	\ar[dd]^{\varphi}\\
    &	f^{-1}\pa{U}	\ar[r]^-{\rest{f}}																																					&	U\\
    \rest{f}^{\ast}\pa{U\times F}	\ar[rrr]	\ar[ru]_{\rest{f}^{\ast}(\pr{1})}																		&	&	&	U\times F  \ar[ul]^{\pr{1}}}
\end{split}\end{equation}
where $\rest{f}^{\ast}\pa{\varphi}$ is a homeomorphism (since $\varphi$ is a homeomorphism).
The subspaces $\rest{f}^{\ast}\pa{p^{-1}\pa{U}}$ and $\fa(p)^{-1}\pa{f^{-1}(U)}$ of $X\times E$ are equal, and the map:
\begin{equation}\begin{split}
\xymatrix@R=0pt{
	\rest{f}^{\ast}\pa{U\times F}	\ar[r]^-{\pi}		&	f^{-1}\pa{U}\times F\\
	\pa{x,\pa{f\pa{x},d}}		\ar@{|-{>}}[r]	&	\pa{x,d}}
\end{split}\end{equation}
is a homeomorphism.
Hence, $\pa{f^{-1}\pa{U},F,\pi\circ \rest{f}^{\ast}(\varphi)}$ is the desired local trivialization of $\fa\pa{p}$ at $x$.
\end{proof}

\begin{example}
Let $f:S^1\to S^1$ be $w\mapsto w^m$ and let $p:S^1\to S^1$ be $z\mapsto z^n$ for fixed natural numbers $m$ and $n$.
Then, $f$, $p$, $\fa\pa{p}$, and $\wt{f}$ are all covering maps,
and $\fa\pa{S^1}=\cpa{\pa{w,z}\mid w^m=z^n}\subset S^1\times S^1$ is a torus link with $g:=\gcd\pa{m,n}$ equally spaced components.
Each component of $\fa\pa{S^1}$ winds $n/g$ times around $S^1\times S^1$ in the $w$ direction and $m/g$ times in the $z$ direction.
Thus, on each component of $\fa\pa{S^1}$, $\fa\pa{p}$ restricts to an $n/g$ fold cover, and $\wt{f}$ restricts to an $m/g$ fold cover.
In particular, if $n|m$, then $\fa\pa{p}$ is a trivial cover. 
\end{example}

\begin{example}
Even for the pullback of a covering map~\eqref{basicpullback}, $\wt{f}$ is generally not open or closed.
For instance, let $f:[0,1)\to\R$ be inclusion and let $p:\R\to\R$ be the identity.
\end{example}

\begin{example}
By Corollary~\ref{embedding}, if $f$ is an embedding in~\eqref{basicpullback}, then $\wt{f}$ is an embedding.
If $f$ is merely a continuous bijection, then $\wt{f}$ is not necessarily an embedding, even with the assumption that $p$ is a covering map.
Consider \hbox{$f:[0,1)\to S^1$} given by $t\mapsto\exp\pa{2\pi it}$ and $p:S^1\to S^1$ given by the identity.
\end{example}

\subsection{Categories of Coverings}\label{ss:catcov}

Let $Y$ be a topological space.
We consider four categories of coverings.
First, the \textbf{objects} of $\cov(Y)$ are arbitrary covering maps $p:E\to Y$,
and a \textbf{morphism} between two such objects is an ordinary morphism of covering maps as defined in~\eqref{morphismofcovers}.
Second, the \textbf{objects} of $\scov(Y)$ are surjective covering maps $p:E\to Y$,
and a \textbf{morphism} between two such objects is a surjective morphism of covering maps.
Third, if $\bs{Y}$ is based, then the \textbf{objects} of $\bcov\bs{Y}$ are based covering maps $p:\bs{E}\to\bs{Y}$,
and a \textbf{morphism} between two such objects is a based morphism of covering maps.
Fourth, the \textbf{objects} of $\bscov\bs{Y}$ are based, surjective covering maps $p:\bs{E}\to\bs{Y}$,
and a \textbf{morphism} between two such objects is a based, surjective morphism of covering maps.\\

In each of these four cases, if $p:E\to Y$ is an object, then the \textbf{identity morphism} is the identity map $1_E:E\to E$.
And, \textbf{composition} of morphisms $t_1:p_1\to p_2$ and $t_2:p_2\to p_3$ is defined to be usual composition of functions $t_2\circ t_1$.
So, composition of morphisms is associative, and the left and right unit laws hold
(i.e., if $p_1:E_1\to Y$ and $p_2:E_2\to Y$ are objects and $t:p_1\to p_2$ is a morphism, then $t\circ 1_{E_1}=t=1_{E_2}\circ t$).
Hence, $\cov(Y)$, $\scov(Y)$, $\bcov\bs{Y}$, and $\bscov\bs{Y}$ are categories.\\

By our convention, when considering $\bcov$ or $\bscov$, a \textbf{map} $f:X\to Y$ means a based map $f:\bs{X}\to\bs{Y}$.
In particular, $y_0=f\pa{x_0}$, $X\neq\emptyset$, $Y\neq\emptyset$, and any object is nonempty.
Basepoints will be implicit at times, especially with objects.
No special assumptions are made in the two unbased categories.
For example, when considering $\scov$, a map $f:X\to Y$ may or may not be surjective.\\

Fix one of the four categories of coverings of a space $Y$.
If $p_1:E_1\to Y$ and $p_2:E_2\to Y$ are objects, then $\h{E_1,E_2}$ and $\h{p_1,p_2}$ both denote the collection of morphisms $p_1\to p_2$.
Further, $t\in\h{E_1,E_2}$ is an \textbf{isomorphism} provided $t$ is a homeomorphism.
We write $E_1\cong E_2$ or $p_1\cong p_2$ to mean there exists an isomorphism $p_1\to p_2$.
An object $p:E\to Y$ is \textbf{trivial} provided there is an isomorphism $p\to\pr{1}$ where $\pr{1}:Y\times D\to Y$ is an object and $D$ is discrete.
In particular, for $\scov$ we have $D\neq\emptyset$ if $Y\neq\emptyset$, and for the two based categories $Y\times D$ is based at $\pa{y_0,d_0}$ for some $d_0\in D$.

\begin{lemma}\label{pullbackfunctor}
Let $f:X\to Y$ be a map. The following are covariant functors:
\begin{enumerate}[label=\tn{(\arabic*)}]\setcounter{enumi}{0}
\item\label{fcov} $\fa:\tn{Cov}\pa{Y}\to\tn{Cov}\pa{X}$.
\item\label{scov} $\fa:\tn{SCov}\pa{Y}\to\tn{SCov}\pa{X}$.
\item\label{fbcov} $\fa:\tn{BCov}\bs{Y}\to\tn{BCov}\bs{X}$.
\item\label{fbscov} $\fa:\tn{BSCov}\bs{Y}\to\tn{BSCov}\bs{X}$.
\end{enumerate}
\end{lemma}

\begin{proof}[Proof of Lemma~\ref{pullbackfunctor}]
Lemma~\ref{pullbackiscover} defines $\fa$ on objects, and Lemma~\ref{pullbacktriangle} defines $\fa$ on morphisms.
If $p:E\to Y$ is surjective, then $\fa\pa{p}$ is surjective by Lemma~\ref{inclusion}.
In the based cases, recall that if $p:\bs{E}\to\bs{Y}$, then $\fa\pa{E}$ is naturally based at $\pa{x_0,e_0}$.
The identity and composition axioms for morphisms follow from Lemma~\ref{pullbacktriangle}.
\end{proof}

\begin{corollary}\label{ipbi}
Fix a category of coverings. If $f:X\to Y$ is a map, $p_1:E_1\to Y$ and $p_2:E_2\to Y$ are objects, and $E_1\cong E_2$, then $\fa\pa{E_1}\cong\fa\pa{E_2}$.
\end{corollary}

\begin{proof}[Proof of Corollary~\ref{ipbi}]
Immediate by functoriality of $\fa$.
\end{proof}

\begin{lemma}\label{trivialcover}
Fix a category of coverings.
Let $f:X\to Y$ be a map.
Let $p$ denote the trivial object \hbox{$\pr{1}:Y\times D\to Y$} where $D$ is discrete.
Let $q$ denote the trivial object $\pr{1}:X\times D\to X$.
Then, $\fa\pa{p}\cong q$.
Specifically, the morphism $\mu:q\to \fa\pa{p}$ given by $\pa{x,d}\mapsto\pa{x,\pa{f(x),d}}$ is an isomorphism.
For the based categories, $\fa\pa{Y\times D}$ is based at $\pa{x_0,\pa{y_0,d_0}}$ and $X\times D$ is based at $\pa{x_0,d_0}$.
\end{lemma}

\begin{proof}[Proof of Lemma~\ref{trivialcover}]
The universal property of pullback yields the morphism $\mu$ by considering the maps $\pr{1}:X\times D\to X$ and $X\times D\to Y\times D$ given by $\pa{x,d}\mapsto\pa{f(x),d}$.
The restriction of the projection $X\times\pa{Y\times D}\to X\times D$ to $\fa\pa{Y\times D}$ is the (continuous) inverse of $\mu$.
\end{proof}

Corollary~\ref{ipbi} and Lemma~\ref{trivialcover} yield the following.

\begin{corollary}\label{pullbacktrivialcover}
Fix a category of coverings.
If $f:X\to Y$ is a map and $p:E\to Y$ is a trivial object, then $\fa\pa{p}$ is a trivial object.
If $X\neq\emptyset$, then fibers of $p$ and of $\fa\pa{p}$ are homeomorphic.
\end{corollary}

\subsection{Disjoint Unions of Covers}\label{ss:ducov}

Given a collection of coverings of a fixed space $Y$, intuitively one would like to stack up the covers, thus yielding one big cover of $Y$.
Conversely, given a cover of $Y$, one may wish to pick out some of the components of the cover thus yielding a smaller cover of $Y$.
Perhaps surprisingly, both of these operations are invalid in complete generality.
With some local niceness hypotheses on $Y$, both operations hold. We explain these facts below.\\ 
 
Recall two different notions of topological disjoint union.
First, suppose that a topological space $Z$ equals the union of open and pairwise disjoint subspaces $Z_\alpha\subset Z$ for $\alpha$ in some index set $I$.
Then, we write $Z=\bigsqcup_{\alpha\in I} Z_{\alpha}$ and say that $Z$ is the \textbf{intrinsic disjoint union} of the subspaces $Z_{\alpha}$.
Note that each $Z_\alpha$ is also closed in $Z$.
Second, let $Z_\alpha$, $\alpha\in I$, be a collection of (not necessarily disjoint) topological spaces where $I$ is some index set.
The \textbf{extrinsic disjoint union} of the $Z_\alpha$ consists of the set $Z := \bigcup_{\alpha\in I} Z_{\alpha} \times\cpa{\alpha}$, denoted $\coprod_{\alpha\in I} Z_{\alpha}$, 
together with the canonical injections $i_\alpha:Z_\alpha\to Z$ where $i_\alpha\pa{z}:=\pa{z,\alpha}$.
The set $Z$ is topologized by: $U\subset Z$ is open if and only if $i_\alpha^{-1}\pa{U}$ is open in $Z_\alpha$ for each $\alpha\in I$.
Each $i_\alpha$ is an open and closed embedding.
The extrinsic disjoint union satisfies a universal property:
given maps with fixed target $f_\alpha:Z_\alpha\to Y$ for $\alpha\in I$,
then there exists a unique map $f:Z\to Y$, denoted $\coprod_{\alpha\in I} f_{\alpha}$, such that $f_\alpha = f\circ i_\alpha$ for each $\alpha\in I$.
Note that each extrinsic disjoint union is an intrinsic disjoint union:
$\coprod_{\alpha\in I} Z_{\alpha}=\bigsqcup_{\alpha\in I} Z_{\alpha}\times\cpa{\alpha}$.

\begin{remark}
In general, an extrinsic disjoint union of based coverings of $\bs{Y}$ does not come equipped with a preferred basepoint.
A similar issue arises when picking out components of a given based cover.
Thus, isomorphisms concerning disjoint unions of coverings are isomorphisms in the category $\cov$ unless explicitly stated otherwise.
\end{remark}

\begin{lemma}\label{comps_cover}
Let $p:E\to Y$ be a covering map where $Y$ is locally connected.
Suppose $E=\bigsqcup_{\alpha\in I}E_\alpha$ for some index set $I$, and
let $J\subset I$ be arbitrary.
Define $g:=\rest{p}\bigsqcup_{\alpha\in J} E_\alpha$.
Then, $g:\bigsqcup_{\alpha\in J}E_\alpha \to Y$ is a covering map.
\end{lemma}

\begin{proof}[Proof of Lemma~\ref{comps_cover}]
As $p$ is a local homeomorphism, $E$ is locally connected and so its components are open and closed in $E$.
Each $E_\alpha$ is open and closed in $E$ and hence is an intrinsic disjoint union of components of $E$.
So, without loss of generality, we assume each $E_\alpha$ is a component of $E$.
Let $y\in Y$ and let $\pa{U,F,\varphi}$ be a local trivialization of $p$ at $y$.
As $Y$ is locally connected, we may assume $U$ is connected.
Define $F':=\cpa{d\in F \mid \varphi^{-1}(y,d)\in \bigsqcup_{\alpha\in J}E_\alpha}$.
Then, $\varphi^{-1}\pa{U\times F'}=g^{-1}\pa{U}$ and $\pa{U,F',\rest{\varphi}}$ is a local trivialization of $g$ at $y$.
\end{proof}

\begin{lemma}\label{imageclosed}
Let $p:E\to Y$ be a covering map where $Y$ is locally connected.
If $C$ is a component of $E$,
then $p\pa{C}$ is closed in $Y$ and, hence, equals a component of $Y$.
\end{lemma}

\begin{proof}[Proof of Lemma~\ref{imageclosed}]
As $p$ is a local homeomorphism, $E$ is locally connected and so $C$ is open and closed in $E$.
As $p$ is open, $p\pa{C}$ is open in $Y$.
Let $y$ be in $\overline{p\pa{C}}$, the closure of $p\pa{C}$ in $Y$.
Let $U$ be a connected, evenly covered neighborhood of $y$ in $Y$.
As $U\cap p\pa{C}\neq\emptyset$, there exists a component $V$ of $p^{-1}\pa{U}$ that intersects $C$.
As $\rest{p}V:V\to U$ is a homeomorphism, $V$ is connected and so $V\subset C$.
It follows that $y\in p\pa{C}$, and so $p\pa{C}$ is closed in $Y$.
\end{proof}

\begin{corollary}\label{decompose_cover_componentwise}
Let $p:E\to Y$ be a covering map where $Y$ is locally connected.
Let $E'\subset E$ be a union of components of $E$.
Then, $p\pa{E'}$ is a union of components of $Y$.
Further, if $Y'\subset Y$ is a union of components of $Y$ such that $p\pa{E'}\subset Y'$,
then $\rest{p}:E'\to Y'$ is a covering map.
\end{corollary}

\begin{proof}[Proof of Corollary~\ref{decompose_cover_componentwise}]
The first conclusion follows by Lemma~\ref{imageclosed}.
Lemma~\ref{comps_cover} implies $\rest{p}:E'\to Y$ is a covering map,
and then Lemma~\ref{restcov} implies $\rest{p}:E'\to Y'$ is a covering map.
\end{proof}

\begin{lemma}\label{cover_component}
Let $Y$ be locally connected and let $B$ be a component of $Y$.
If $p:E\to B$ is a covering map, then $q:E\to Y$ given by $q(e):=p(e)$ is a covering map.
\end{lemma}

\begin{proof}[Proof of Lemma~\ref{cover_component}]
Components of $Y$ are open and closed in $Y$.
So, any local trivialization $\pa{U,F,\varphi}$ of $p$ at $y\in B$ is also a local trivialization of $q$ at $y\in Y$.
For $y\in Y-B$, $\pa{Y-B,\emptyset,\emptyset}$ is a local trivialization of $q$ at $y$.
\end{proof}

\begin{lemma}\label{union_covers}
Let $Y$ be locally path-connected and semilocally simply-connected.
For each $\alpha$ in some index set $I$, let $p_\alpha:E_\alpha \to Y$ be a covering map.
Define \hbox{$E:=\coprod_{\alpha\in I} E_\alpha$} and $p:=\coprod_{\alpha\in I} p_\alpha$.
Then, $p:E\to Y$ is a covering map.
\end{lemma}

\begin{proof}[Proof of Lemma~\ref{union_covers}]
Let $y\in Y$. Let $U$ be an open neighborhood of $y$ in $Y$ such that $i_\sharp:\pi_1\pa{U,y}\to\pi_1\pa{Y,y}$ is trivial.
Replacing $U$ with its path component containing $y$, we can and do further assume $U$ is path-connected.
For each $\alpha\in I$ and $e\in p_{\alpha}^{-1}(y)$, there exists a unique lift $h_{\alpha,e}$ of $i:U\to Y$ to $E_\alpha$ such that $h_{\alpha,e}\pa{y}=e$ by Propositions~1.33 and~1.34 of~\cite{hatcher}.
The image of $h_{\alpha,e}$ contains a unique point, $e$, of $p_\alpha^{-1}\pa{y}$ since $i$ is injective.
Further, lifts $h_{\alpha,e}$ and $h_{\alpha,e'}$ for distinct points $e$ and $e'$ in the fiber $p_\alpha^{-1}\pa{y}$ have disjoint images in $E_\alpha$ by homotopy lifting~\cite[Prop.~1.30]{hatcher} (since $i_\sharp$ is trivial).
We claim $h_{\alpha,e}$ is an open map.
Let $u\in U$. Let $W\subset U$ be a path-connected, open neighborhood of $u$ in $Y$ that is evenly covered by $p_\alpha$.
There exists a unique component $C$ of $p_\alpha^{-1}\pa{W}$ that is contained in the image of $h_{\alpha,e}$.
Evidently, $h_{\alpha,e}\pa{W}=C$. Thus, $h_{\alpha,e}$ is open as claimed.
Hence, each $h_{\alpha,e}$ is an embedding.
Unique path lifting readily implies that if $z\in p_\alpha^{-1}(U)$, then $z\in\im{h_{\alpha,e}}$ for a unique $e\in p_\alpha^{-1}(y)$;
this observation yields the map $\delta:p_\alpha^{-1}(U)\to p_\alpha^{-1}(y)$ given by $z\mapsto e$.
Combining these observations, we get:
\[
	p_\alpha^{-1}\pa{U}=\bigsqcup_{e\in p_\alpha^{-1}(y)}\im{h_{\alpha,e}}.
\]
Define the map $\varphi_\alpha:p_\alpha^{-1}(U)\to U\times p_\alpha^{-1}(y)$ by $\varphi_\alpha(z):=\pa{p_\alpha(z),\delta(z)}$.
Evidently, $\psi_\alpha:U\times p_\alpha^{-1}(y) \to p_\alpha^{-1}(U)$ given by $\psi_\alpha\pa{u,e}:=h_{\alpha,e}(u)$ is the continuous inverse of $\varphi_\alpha$.
Thus, $\pa{U,p_{\alpha}^{-1}(y),\varphi_{\alpha}}$ is a local trivialization of $p_\alpha$ at $y$.
Define the homeomorphism:
\begin{equation*}\begin{split}
\xymatrix@R=0pt{
	p^{-1}\pa{U}	\ar[r]^-{\psi}	&	U\times\coprod_{\alpha\in I} p_\alpha^{-1}\pa{y}\\
	\pa{z,\alpha}		\ar@{|-{>}}[r]	&	\pa{p_\alpha\pa{z}, i_\alpha\circ\pr{2}\circ\varphi_\alpha\pa{z}}}
\end{split}\end{equation*}
Then, $\pa{U,\coprod_{\alpha\in I} p_{\alpha}^{-1}(y),\psi}$ is a local trivialization of $p$ at $y$.
\end{proof}

Without some other additional restrictions, Lemmas~\ref{comps_cover} and~\ref{union_covers} both become false if any of the local niceness hypotheses on $Y$ are omitted, as shown by the next three examples.

\begin{example}
Let $Y:=\cpa{0}\cup\cpa{1/k \mid k\in\N}\subset\R$ and let $\pr{1}:Y\times\N_{0}\to Y$, which is a trivial cover. Define:
\begin{align*}
	E_{1}	&:=	\cpa{\pa{0,n}\mid n\in\N} \cup \cpa{\pa{1/k,n}\mid k\in\N \; \tn{\&} \; 1\leq n \leq k} \; \tn{and}\\
	E_{2}	&:=	\pa{Y\times\N_{0}} - E_{1}.
\end{align*}
Notice that $Y\times\N_{0}=E_{1}\sqcup E_{2}$.
Neither of the (surjective) maps \hbox{$\rest{\pr{1}}:E_{1}\to Y$} nor $\rest{\pr{1}}:E_{2}\to Y$ is a covering map (consider cardinalities of fibers).
\end{example}

\begin{example}
Let $Y:=\cpa{0}\cup\cpa{1/k \mid k\in\N}\subset\R$.
Let $D$ be an uncountable, discrete space and fix some $d'\in D$.
Let $W:=Y\times D\times\N$ and write $\pr{1}:W\to Y$.
For each $n\in\N$, define:
\[
	E_n:=\pa{ Y\times\cpa{d'}\times\cpa{n} } \cup  \pa{ \cpa{1/n}\times D\times\cpa{n} } \subset W
\]
and let $p_n$ denote the (nontrivial) covering map $\rest{\pr{1}}: E_{n} \to Y$.
The map $\coprod_{n\in\N} p_n$ is not a covering map (again, consider cardinalities of fibers).
\end{example}

\begin{example}\label{he_ex}
Let $Y\subset\R^2$ be the union of mutually tangent circles $C_n$, $n\in \N$, where $C_n$ has radius $1/n$ (see Figure~\ref{earring}).
\begin{figure}[h!]
    \centerline{\includegraphics{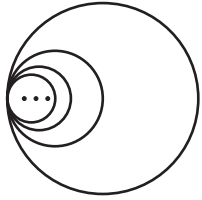}}
    \caption{Hawaiian earring $Y\subset\R^2$.}
    \label{earring}
\end{figure}
The space $Y$ is the well known \emph{Hawaiian earring} and is not semilocally simply-connected.
Let $y_0\in Y$ denote the wild point where the circles intersect.
In each $C_n$, let $y_n$ denote point antipodal to $y_0$.
For each $n\in\N$, let $E_n$ be the space obtained as follows:
begin with the disjoint union of two copies of $Y$,
cut $C_n$ at $y_n$ in each copy of $Y$ leaving four loose strands,
finally glue the four endpoints of the loose strands together in the obvious way thus obtaining a connected, double cover $p_n:E_{n}\to Y$.
The map $\coprod_{n\in\N} p_n$ is not a covering map (there is no local trivialization at $y_0\in Y$).
\end{example}

\begin{lemma}\label{isodu}
Let $Y$ be a topological space.
For each $\alpha$ in some index set $I$, let $p_\alpha:E_\alpha \to Y$ and $p'_\alpha:E'_\alpha \to Y$ be covering maps.
Let $E:=\coprod_{\alpha\in I}E_\alpha$ and let $p$ denote the map $\coprod_{\alpha\in I}p_\alpha:E\to Y$ (not necessarily a covering map).
Similarly, let $E':=\coprod_{\alpha\in I}E'_\alpha$ and let $p'$ denote the map $\coprod_{\alpha\in I}p'_\alpha:E'\to Y$.
Suppose that $E_\alpha\cong E'_\alpha$ for each $\alpha\in I$.
Then, there is a homeomorphism $t:E\to E'$ such that $p=p'\circ t$.
In particular, if $p$ or $p'$ is a covering map, then $p$ and $p'$ are both covering maps and $p\cong p'$ (isomorphism in $\cov\pa{Y}$).
\end{lemma}

\begin{proof}[Proof of Lemma~\ref{isodu}]
For each $\alpha\in I$, let $i_\alpha:E_\alpha\to E$ and $i'_\alpha:E'_\alpha\to E'$ be the canonical injections,
and let $t_\alpha:E_\alpha\to E'_\alpha$ be a homeomorphism such that $p_\alpha=p'_\alpha\circ t_\alpha$.
Hence, $i'_\alpha\circ t_\alpha:E_\alpha\to E'$ for each $\alpha\in I$ and the universal property of disjoint union yields the map $t:E\to E'$.
Evidently, $t\pa{e,\alpha}=\pa{t_\alpha\pa{e},\alpha}$.
It is easy to check that $t$ is a bijection and $p=p'\circ t$.
To show $t$ is open, it suffices to show each $i'_\alpha \circ t_\alpha$ is open. But, this is immediate since $t_\alpha$ is a homeomorphism and $i'_\alpha$ is open.
The final conclusion follows by Lemma~\ref{homeotocover}.
\end{proof}

\subsection{Pullback of Disjoint Unions of Covers}\label{ss:pbducov}

\begin{lemma}\label{inverseopenclosed}
Given a diagram of maps: 
\begin{equation}\begin{split}\label{basicsubset}
\xymatrix{
    							&	Z	\ar[d]^-{g}\\
    X	\ar[r]^{f}	&	Y}
\end{split}\end{equation}
If $S\subset Z$ is open (respectively closed), then $\fa\pa{S}$ is open (respectively closed) in $\fa\pa{Z}$,
and $\fa\pa{\rest{g}S}=\rest{\fa\pa{g}}\fa\pa{S}$.
\end{lemma}

\begin{proof}[Proof of Lemma~\ref{inverseopenclosed}]
Let $i:S\to Z$ be inclusion.
By Lemma~\ref{pullbacktriangle}, pullback yields the commutative diagram:
\begin{equation}\label{subset_pullback}\begin{split}
\xymatrix{
    \fa\pa{S}	\ar[rrr]	\ar[rd]^{\fa\pa{\rest{g}S}}	\ar[dd]_{\fa\pa{i}}	&	&	&	S  \ar[dl]_{\rest{g}S}	\ar[dd]^{i}\\
    &	X	\ar[r]^-{f}																																					&	Y\\
    \fa\pa{Z}	\ar[rrr]^{\wt{f}=h}	\ar[ru]_{\fa\pa{g}}																		&	&	&	Z  \ar[ul]^{g}}
\end{split}\end{equation}
where $\fa\pa{i}$ is inclusion. Let $h$ denote the map $\wt{f}:f^{\ast}\pa{Z}\to Z$.
The result follows since $h^{-1}\pa{S}=f^{\ast}\pa{S}$ as subsets of $f^{\ast}\pa{Z}$,
and by commutativity of~\ref{subset_pullback}.
\end{proof}

\begin{proposition}[\textbf{Pullback of intrinsic disjoint union}]\label{pbi}
Let $f:X\to Y$ be a given map where $Y$ is locally connected.
Let $p:E\to Y$ be a covering map. 
Suppose that $E=\bigsqcup_{\alpha\in I} E_\alpha$ for some index set $I$.
Then:
\begin{equation}\label{pbidu}
	\fa\pa{\bigsqcup_{\alpha\in I} E_\alpha} = \bigsqcup_{\alpha\in I} \fa\pa{E_\alpha}
\end{equation}
and, for each $\alpha\in I$, the following is a covering map:
\begin{equation}\label{pbidumap}
	\rest{\fa\pa{p}}\fa\pa{E_\alpha}=\fa\pa{\rest{p}E_\alpha}:\fa\pa{E_\alpha}\to X.
\end{equation}
\end{proposition}

\begin{proof}[Proof of Proposition~\ref{pbi}]
For each $\alpha\in I$, the map $\rest{p}E_\alpha:E_\alpha\to Y$ is a covering map by Lemma~\ref{comps_cover}, and hence
$\fa\pa{\rest{p}E_\alpha}:\fa\pa{E_\alpha}\to X$ is a covering map by Lemma~\ref{pullbackiscover}.
Property~\eqref{pbidu} follows by the definition of pullback~\eqref{pullback_set} and by Lemma~\ref{inverseopenclosed}.
The equality in~\eqref{pbidumap} is immediate by Lemma~\ref{inverseopenclosed}. 
\end{proof}

\begin{proposition}[\textbf{Pullback of extrinsic disjoint union}]\label{pbe}
Let $f:X\to Y$ be a given map where 
 $Y$ is locally path-connected and semilocally simply-connected.
For each $\alpha$ in some index set $I$, let $p_\alpha:E_\alpha\to Y$ be a covering map.
Then, there is an isomorphism (typically not an equality) in $\cov\pa{X}$:
\begin{equation}
	\fa\pa{\coprod_{\alpha\in I} E_\alpha} \cong \coprod_{\alpha\in I}\fa\pa{E_\alpha}.
\end{equation}
In particular, the map $\coprod_{\alpha\in I}\fa\pa{p_\alpha}:\coprod_{\alpha\in I}\fa\pa{E_\alpha} \to X$ is a covering map.
\end{proposition}

\begin{proof}[Proof of Proposition~\ref{pbe}]
Throughout this proof, all disjoint unions are over $\alpha\in I$.
Let $E:=\coprod E_\alpha$ and $p:=\coprod p_\alpha$.
By Lemma~\ref{union_covers}, $p:E\to Y$ is a covering map.
Hence, $\fa\pa{p}:\fa\pa{E}\to X$ is a covering map by Lemma~\ref{pullbackiscover}.\\

For each $\alpha\in I$, $\fa\pa{p_\alpha}:\fa\pa{E_\alpha}\to X$ is a covering map by Lemma~\ref{pullbackiscover},
and the canonical injection $i_\alpha:E_\alpha\to E$ satisfies $p_\alpha=p\circ i_\alpha$.
By Lemma~\ref{pullbacktriangle}, we get the commutative diagram:
\begin{equation}\begin{split}\label{pullback_can_inj}
\xymatrix{
    \fa\pa{E_\alpha}	\ar[rrr]	\ar[rd]^-{\fa\pa{p_\alpha}}	\ar[dd]_{\fa\pa{i_\alpha}}	&	&	&	E_\alpha  \ar[dl]_{p_\alpha}	\ar[dd]^{i_\alpha}\\
    &	X	\ar[r]^{f}																																								&	Y\\
    \fa\pa{E}	\ar[rrr]	\ar[ru]_-{\fa\pa{p}}															&	&	&	E  \ar[ul]^{p}}
\end{split}\end{equation}

Consider the disjoint union $\coprod\fa\pa{E_\alpha}$, which is not yet known to be a cover of $X$.
For each $\alpha\in I$, we have the canonical injection \hbox{$j_\alpha:\fa\pa{E_\alpha}\to \coprod\fa\pa{E_\alpha}$}.
By the universal property of disjoint union, there exists a unique map \hbox{$g:\coprod\fa\pa{E_\alpha}\to X$} such that $\fa\pa{p_\alpha}=g\circ j_\alpha$ for each $\alpha\in I$. Evidently, $g\pa{\pa{x,e},\alpha}=x$.\\

A second application of the universal property of disjoint union yields the unique map $t:\coprod\fa\pa{E_\alpha}\to\fa\pa{E}$ such that 
$\fa\pa{i_\alpha}=t\circ j_\alpha$ for each $\alpha\in I$. Evidently, $t\pa{\pa{x,e},\alpha}=\pa{x,\pa{e,\alpha}}$.\\

We now have the commutative diagram of maps:
\begin{equation}\begin{split}\label{homeotocoverdiag2}
\xymatrix{
    \coprod\fa\pa{E_\alpha}	\ar[rr]^-{t}	\ar[dr]_{g}	&	&	\fa\pa{E}	\ar[dl]^{\fa\pa{p}}\\
    &	X }
\end{split}\end{equation}
where $\fa\pa{p}$ is a covering map.
By Lemma~\ref{homeotocover}, it suffices to show that $t$ is a homeomorphism.
A straightforward exercise shows that $t$ is a bijection.
To show that $t$ is open, it suffices to show that each map $\fa\pa{i_\alpha}:\fa\pa{E_\alpha}\to\fa\pa{E}$ is open.
But, each canonical injection $i_\alpha$ is open, and so $\fa\pa{i_\alpha}$ is open by Lemma~\ref{pullbacktriangle}.
\end{proof}

\begin{proposition}[\textbf{Pullback of partitioned extrinsic disjoint union}]\label{pbpdu}
Let $f:X\to Y$ be a given map where 
 $Y$ is locally path-connected and semilocally simply-connected.
For each $\alpha$ in some index set $I$, let $p_\alpha:E_\alpha\to Y$ be a covering map.
Let $E:=\coprod_{\alpha\in I}E_\alpha$ and let $p$ denote the (covering) map $\coprod_{\alpha\in I}p_\alpha:E\to Y$.
Let $I=\bigsqcup_{\beta\in J}I_\beta$ be an intrinsic disjoint union of sets where $J$ is some index set.
Then, the following is an isomorphism in $\cov\pa{Y}$:
\begin{equation}\label{Ydu}
	E \cong \coprod_{\beta\in J}\coprod_{\alpha\in I_{\beta}}E_\alpha
\end{equation}
and the following are isomorphisms in $\cov\pa{X}$:
\begin{equation}\label{Xdu}
\fa\pa{E} \cong \fa\pa{\coprod_{\beta\in J}\coprod_{\alpha\in I_{\beta}}E_\alpha} \cong
 \coprod_{\beta\in J}\fa\pa{\coprod_{\alpha\in I_{\beta}}E_\alpha} \cong \coprod_{\beta\in J}\coprod_{\alpha\in I_{\beta}} \fa\pa{E_\alpha}.
\end{equation}
\end{proposition}

\begin{proof}[Proof of Proposition~\ref{pbpdu}]
The map $p:E\to Y$ is a covering map by Lemma~\ref{union_covers}.
For each $\beta\in J$, the map $\coprod_{\alpha\in I_\beta}p_\alpha:\coprod_{\alpha\in I_{\beta}}E_\alpha\to Y$, denoted $\pi_\beta$, is a covering map by Lemma~\ref{union_covers}.
Hence, the map $\coprod_{\beta\in J}\pi_\beta:\coprod_{\beta\in J}\coprod_{\alpha\in I_{\beta}}E_\alpha \to Y$, denoted $\pi$, is a covering map by Lemma~\ref{union_covers}.
Evidently, $\pi\pa{\pa{e,\alpha},\beta}=p_\alpha\pa{e}$.
By Lemma~\ref{pullbackiscover}, each of the maps $\fa\pa{p}$, $\fa\pa{\pi_\beta}$ where $\beta\in J$, and $\fa\pa{\pi}$ is a covering map with target $X$.\\

The canonical injections $i_\alpha:E_\alpha \to E$, $\alpha \in I$, yield the maps $j_\beta:\coprod_{\alpha\in I_{\beta}}E_\alpha \to E$ for $\beta\in J$ by the universal property of disjoint union. In turn, the $j_\beta$ yield the map $t:\coprod_{\beta\in J}\coprod_{\alpha\in I_{\beta}}E_\alpha \to E$.
Evidently, $t\pa{\pa{e,\alpha},\beta}=\pa{e,\alpha}$. We leave the reader the easy verification that $t$ is a homeomorphism and $\pi=p\circ t$. This proves~\eqref{Ydu}.\\

The first isomorphism in~\eqref{Xdu} follows from~\eqref{Ydu} and Corollary~\ref{ipbi}.
The second isomorphism in~\eqref{Xdu} follows from Proposition~\ref{pbe} applied to the covering maps $\pi_\beta$, $\beta\in J$.
In particular, the map $h:=\coprod_{\beta\in J}\fa\pa{\pi_\beta}$ is a covering map with target $X$.
For each $\beta\in J$, the maps $p_\alpha$ for $\alpha\in I_\beta$ are covering maps.
So, Proposition~\ref{pbe} implies that $\coprod_{\alpha\in I_\beta}\fa\pa{p_\alpha}$ is a covering map
and $\coprod_{\alpha\in I_\beta}\fa\pa{E_\alpha}\cong\fa\pa{\coprod_{\alpha\in I_\beta}E_\alpha}$.
Thus, Lemma~\ref{isodu} implies the third isomorphism in~\eqref{Xdu} since $h$ is already known to be a covering map.
\end{proof}

\begin{remark}
Note that in Propositions~\ref{pbe} and~\ref{pbpdu}, no local niceness hypotheses on $X$ were necessary.
\end{remark}

\section{Generalizing Quillen's triad}\label{s:gqt}

\subsection{Generalizing Quillen's~\ref{q1}}\label{ss:gq1}

If $X$ is a topological space, then $\Gamma\pa{X}$ denotes the set of components of $X$.
By definition, each component of $X$ is nonempty, although $\Gamma\pa{X}$ itself is nonempty if and only if $X$ is nonempty.
A map $f:X\to Y$ induces the function $\fs:\Gamma\pa{X}\to\Gamma\pa{Y}$ given by $\br{x}\mapsto\br{f\pa{x}}$.
The following is our generalization of~\ref{q1}.
Recall that $\fa$ is one of the four functors in Lemma~\ref{pullbackfunctor}.

\begin{proposition}\label{gq1}
Fix a category of coverings.
Let $f:X\to Y$ be a map where $Y$ is locally connected.
Then, $\fs:\Gamma\pa{X}\to\Gamma\pa{Y}$ is surjective if and only if $\fa$ is faithful.
\end{proposition}

\begin{proof}[Proof of Proposition~\ref{gq1}]
The result is vacuously true if $Y$ is empty, so assume $Y\neq\emptyset$.
We begin with the backward implication.
Suppose, by way of contradiction, that $C$ is a component of $Y$ disjoint from $\im{f}$.
Consider the trivial object $\pr{1}:Y\times\cpa{1,2}\to Y$.
For the the based categories, base $Y\times\cpa{1,2}$ at $\pa{y_0,1}$ and note that $y_0=f\pa{x_0}\notin C$.
Let $t_1$ be the identity morphism $\pr{1}\to \pr{1}$.
Let $t_2$ be the function $Y\times\cpa{1,2}\to Y\times\cpa{1,2}$ which swaps the components $C\times\cpa{1}$ and $C\times\cpa{2}$ and is the identity otherwise.
As $Y$ is locally connected, we see that $t_2$ is an isomorphism in $\h{\pr{1},\pr{1}}$.
As $\im{f}\cap C$ is empty, it is straightforward to check that $\fa\pa{t_1}=\fa\pa{t_2}$.
As $t_1\neq t_2$, $\fa$ is not faithful.\\

For the forward implication, let $p_1:E_1\to Y$ and $p_2:E_2\to Y$ be objects.
Let $t_1,t_2\in\h{p_1,p_2}$ such that $\fa\pa{t_1}=\fa\pa{t_2}$. We must show $t_1=t_2$.
By Corollary~\ref{equalizer_morphisms}, the equalizer of $t_1$ and $t_2$ is open and closed in $E_1$.
Let $C$ be a component of $E_1$. It suffices to prove that $t_1$ and $t_2$ agree at some point of $C$.
By Lemma~\ref{imageclosed}, $p_1\pa{C}$ is a component of $Y$.
By hypothesis, there is a point $y\in\im{f}\cap p_1\pa{C}$, say $y=f\pa{x}=p_1\pa{e}$ where $x\in X$ and $e\in C$.
Then, $\pa{x,e}\in\fa\pa{E_1}$ and:
\[
	\pa{x,t_1\pa{e}}=\fa\pa{t_1}\pa{x,e}=\fa\pa{t_2}\pa{x,e}=\pa{x,t_2\pa{e}}.
\]
So, $t_1\pa{e}=t_2\pa{e}$ and the proof is complete.
\end{proof}

The previous proof utilized local connectivity of $Y$ in both implications.
The following example shows that the backward implication is false in general without this hypothesis.
It is not clear to us if the forward implication holds without this hypothesis.

\begin{example}
Let $X:=\cpa{1/n \mid n\in\N}\subset\R$, let $Y:=\cpa{0}\cup X \subset \R$, and let \hbox{$f:X\to Y$} be inclusion.
In the based categories, base $X$ and $Y$ at $1$.
Let \hbox{$p_1:E_1\to Y$} and $p_2:E_2\to Y$ be objects.
Let $t_1,t_2\in\h{p_1,p_2}$ such that $\fa\pa{t_1}=\fa\pa{t_2}$.
Let $h$ denote the map $\wt{f}:\fa\pa{E_1}\to E_1$.
As \hbox{$\im{h}=p_1^{-1}\pa{X}$}, we get that $t_1=t_2$ on $p_1^{-1}\pa{X}$.
Suppose $e\in p_1^{-1}\pa{0}$. Using a local trivialization of $p_1$ at $0\in Y$, ones sees that $e$ is a limit point of $p_1^{-1}\pa{X}$.
The equalizer of $t_1$ and $t_2$ is closed in $E_1$ by Corollary~\ref{equalizer_morphisms}, and so $t_1\pa{e}=t_2\pa{e}$.
Hence, $t_1=t_2$ on all of $E_1$ and $\fa$ is faithful, even though $\im{f}$ misses the component $\cpa{0}$ of $Y$.
\end{example}

Proposition~\ref{gq1} yields the following alternative generalization of~\ref{q1}.

\begin{corollary}\label{gq1_pc}
Fix a category of coverings.
Let $f:X\to Y$ be a map where $Y$ is locally path-connected.
Then, $\fs:\pi_0\pa{X}\to\pi_0\pa{Y}$ is surjective if and only $\fa$ is faithful.
\end{corollary}

\subsection{Pullback and the Fundamental Group}\label{ss:pb_fg}

Our generalizations of~\ref{q2} and \ref{q3} further utilize the fundamental group.
This subsection collects some basic facts relating pullback and the fundamental group.
The following lemma determines the fundamental group of a based component of the pullback of a covering map.

\begin{lemma}\label{key}
Consider a based pullback diagram of maps:
\begin{equation}\label{based_pullback}\begin{split}
\xymatrix{
    \pa{\fa\pa{E},\pa{x_0,e_0}}	\ar[r]^-{\wt{f}}	\ar[d]_{\fa(p)}	&	\bs{E}	\ar[d]^{p}\\
    \bs{X}  							\ar[r]^-{f}     																& \bs{Y} }
\end{split}\end{equation}
where $p$ is a covering map and $E$ is not necessarily connected.
Let $z_1:=\pa{x_0,e_1}$ be an arbitrary point in the fiber $\fa\pa{p}^{-1}\pa{x_0}$, and let $Z$ be the component of $\fa\pa{E}$ containing $z_1$.
Then, the induced homomorphisms of fundamental groups satisfy:
\begin{equation}\label{mcc}
	\fa\pa{p}_{\sharp}\pa{\pi_{1}\pa{Z,z_{1}}} = \fs^{-1}\pa{p_{\sharp}\pa{\pi_{1}\pa{E,e_{1}}}}.
\end{equation}
\end{lemma}

\begin{proof}
Commutativity of~\eqref{based_pullback} yields ``$\subset$'' in~\eqref{mcc}.
Next, let $\alpha$ be a loop in $X$ based at $x_0$ and satisfying:
\begin{equation}\label{alpha_cond}
	\fs\pa{\br{\alpha}}	=	\br{f\circ\alpha}	\in	p_{\sharp}\pa{\pi_{1}\pa{E,e_{1}}}.
\end{equation}
By path lifting~\cite[Prop.~1.30]{hatcher}, there is a unique lift $\wt{\alpha}:\pa{[0,1],0}\to\pa{E,e_1}$ such that $p\circ\wt{\alpha}=f\circ\alpha$.
By~\eqref{alpha_cond} and homotopy lifting~\cite[Prop.~1.31]{hatcher}, $\wt{\alpha}$ is a loop based at $e_1$.
The universal property of pullback yields the unique map $\mu:[0,1]\to\fa\pa{E}$ such that the induced diagram commutes.
Recall that $\mu(t)=\pa{\alpha(t),\wt{\alpha}(t)}$.
So, $\mu$ is a loop based at $z_1$ and $\im{\mu}$ lies in $Z$.
Thus:
\[
	\fa(p)_{\sharp}\pa{\br{\mu}} = \br{\fa(p)\circ\mu} = \br{\alpha}
\]
and the proof is complete.
\end{proof}

\begin{lemma}\label{comp_bp}
Let $f:X\to Y$ be a map where $X$ and $Y$ are locally path-connected.
Let $p:E\to Y$ be a covering map such that $p\pa{E}$ is connected.
Assume that $f\pa{x_1}=p\pa{e_1}$ for some $x_1\in X$ and $e_1\in E$, and define $y_1:=f\pa{x_1}$.
Assume \hbox{$\fs:\pi_0\pa{X}\to\pi_0\pa{Y}$} is injective.
If $Z$ is a component of $\fa\pa{E}$,
then $Z$ contains a point in the fiber $\cpa{x_1}\times p^{-1}\pa{y_1}$ and $\fa\pa{p}\pa{Z}$ equals the component of $X$ containing $x_1$.
\end{lemma}

\begin{proof}[Proof of Lemma~\ref{comp_bp}]
Let $Z$ be a component of $\fa\pa{E}$, and let $(x,e)\in Z$.
As $Y$ is locally path-connected, $p\pa{E}$ is a path component of $Y$.
Thus, $f(x)$ and $f(x_1)$ lie in the same path component of $Y$.
By hypothesis, $x$ and $x_1$ lie in the same path component of $X$.
Lift a path from $x$ to $x_1$ in $X$ to $\fa\pa{E}$, beginning at $(x,e)$, and the first conclusion follows.
The second conclusion now follows by Lemma~\ref{imageclosed} since $\fa\pa{p}$ is a covering map.
\end{proof}

\begin{lemma}\label{pb_connected}
Let $f:X\to Y$ be a map where $X$ and $Y$ are locally path-connected.
Let $p:E\to Y$ be a covering map where $E$ is connected.
Assume $\fs:\pi_0\pa{X}\to\pi_0\pa{Y}$ is injective and $\fs:\pi_1\pa{X,x}\to\pi_1\pa{Y,f(x)}$ is surjective for each $x\in X$.
Then, $\fa\pa{E}$ is connected.
\end{lemma}

\begin{proof}[Proof of Lemma~\ref{pb_connected}]
If $\fa\pa{E}$ is empty, the result holds.
So, let $\pa{x_1,e_1}\in\fa\pa{E}$.
Define $y_1:=f\pa{x_1}$.
By Lemma~\ref{comp_bp}, it suffices to show that all points in the set \hbox{$\cpa{x_1}\times p^{-1}\pa{y_1}$} lie in the same path component of $\fa\pa{E}$.
Let $\pa{x_1,e_2}\in \cpa{x_1}\times p^{-1}\pa{y_1}$. As $E$ is path connected, there is path $\alpha$ from $e_1$ to $e_2$ in $E$.
Thus, $p\circ\alpha$ is a loop in $Y$ based at $y_1$.
By hypothesis, there is a loop $\beta$ in $X$ based at $x_1$ such that $\fs\pa{\br{\beta}}=\br{p\circ\alpha}$.
In particular, $f\circ\beta$ and $p\circ\alpha$ are path-homotopic.
By path lifting~\cite[Prop.~1.30]{hatcher}, we get $\wt{\beta}$ a path in $\fa\pa{E}$ beginning at $(x_1,e_1)$ and so $\beta=\fa\pa{p}\circ\wt{\beta}$.
Thus, $\wt{f}\circ\wt{\beta}$ is a lift of $f\circ\beta$ to $E$ beginning at $e_1$. Also, $\alpha$ is a lift of $p\circ\alpha$ to $E$ beginning at $e_1$.
As $f\circ\beta$ and $p\circ\alpha$ are path-homotopic, homotopy lifting~\cite[Prop.~1.30]{hatcher} implies that $\wt{f}\circ\wt{\beta}\pa{1}=\alpha(1)=e_2$.
But, $\wt{f}\circ\wt{\beta}\pa{1}=e_2$ means $\wt{\beta}(1)=(x_1,e_2)$.
The proof is complete.
\end{proof}

\begin{remark}
It is easy to construct examples that show Lemmas~\ref{comp_bp} and~\ref{pb_connected} become false when $\fs:\pi_0\pa{X}\to\pi_0\pa{Y}$ is not injective or $\fs:\pi_1\pa{X,x}\to\pi_1\pa{Y,y}$ is not surjective.
\end{remark}

\begin{lemma}\label{basic_group_iso}
Let $X$ be locally path-connected.
Let $C$ be a component of $X$ and let $x_1\in C$.
Let $p_1:\pa{E_1,e_1}\to \pa{X,x_1}$ and $p_2:\pa{E_2,e_2}\to \pa{X,x_1}$ be covering maps where $E_1$ and $E_2$ are connected and $\im{(p_1)_\sharp}=\im{(p_2)_\sharp}$.
Then, $p_1$ and $p_2$ are isomorphic objects in $\bcov\pa{X,x_1}$.
\end{lemma}

\begin{proof}[Proof of Lemma~\ref{basic_group_iso}]
The restrictions $q_1:E_1\to C$ and $q_2:E_2\to C$ are covering maps by Corollary~\ref{decompose_cover_componentwise}.
Next, $q_1\cong q_2$ by~\cite[Prop.~1.37]{hatcher}.
This yields an isomorphism $p_1\to p_2$.
\end{proof}

\subsection{Generalizing Quillen's~\ref{q2}}\label{ss:gq2}

The following is our generalization of~\ref{q2}.

\begin{proposition}\label{gq2}
Fix a category of coverings.
Let $f:X\to Y$ be a map where $X$ and $Y$ are locally path-connected and $Y$ is semilocally simply-connected.
Then, $\fs:\pi_0\pa{X}\to\pi_0\pa{Y}$ is a bijection and $\fs:\pi_1\pa{X,x}\to\pi_1\pa{Y,f\pa{x}}$ is surjective for each $x\in X$ if and only if $\fa$ is full and faithful.
\end{proposition}

\begin{proof}[Proof of Proposition~\ref{gq2}]
We begin with the backward implication.
Corollary~\ref{gq1_pc} implies that $\fs:\pi_0\pa{X}\to\pi_0\pa{Y}$ is surjective.
Suppose, by way of contradiction, that \hbox{$\fs:\pi_0\pa{X}\to\pi_0\pa{Y}$} is not injective.
Let $C_1$ and $C_2$ be distinct path components of $X$ that are sent by $f$ into the same path component, $B$, of $Y$.
For the based categories, we can and do interchange the roles of $C_1$ and $C_2$ if necessary so that $x_0\notin C_1$.
Let $p$ denote the trivial object $\pr{1}:Y\times\cpa{1,2}\to Y$ (base at $\pa{y_0,1}$).
Let $q$ denote the trivial object $\pr{1}:X\times\cpa{1,2}\to X$ (base at $\pa{x_0,1}$).
By Lemma~\ref{trivialcover}, $\varphi:\fa\pa{p}\to q$ given by $\pa{x,\pa{f(x),d}}\mapsto\pa{x,d}$ is an isomorphism.
As $X$ is locally path-connected, the function $\sigma:X\times\cpa{1,2}\to X\times\cpa{1,2}$ which swaps the components $C_1\times\cpa{1}$ and $C_1\times\cpa{2}$, and is the identity otherwise, is an isomorphism.
So, $s:=\varphi^{-1}\circ\sigma\circ\varphi$ is a isomorphism in $\h{\fa\pa{p},\fa\pa{p}}$.
As $\fa$ is full, there exists $t\in\h{p,p}$ such that $\fa\pa{t}=s$.
But, $\sigma$ is the identity on $C_2\times\cpa{1}$, so Corollary~\ref{equalizer_morphisms} implies that $t$ is the identity on $B\times\cpa{1}$.
Thus, $\sigma$ is the identity on $C_1\times\cpa{1}$, a contradiction.
Hence, \hbox{$\fs:\pi_0\pa{X}\to\pi_0\pa{Y}$} is injective.\\

Suppose, by way of contradiction, that $\fs:\pi_1\pa{X,x_1}\to\pi_1\pa{Y,y_1}$ is not surjective where $x_1\in X$ and $y_1:=f\pa{x_1}$.
Let $C$ be the component of $X$ containing $x_1$ and let $B$ be the component of $Y$ containing $y_1$.
Let $g:\pa{E,e_1}\to\pa{B,y_1}$ be a connected cover such that $\im{g_{\sharp}}=\im{\fs}$ (see~\cite[pp.~66--68]{hatcher}).
Assume first that the category of coverings is $\cov$.
By Lemma~\ref{cover_component}, $p:E\to Y$ given by $p(e):=g(e)$ is an object.
Consider the object $\fa\pa{p}:\fa\pa{E}\to X$.
Define $F:=p^{-1}\pa{y_1}$ and note that $\card{F}\geq2$ since $\fs:\pi_1\pa{X,x_1}\to\pi_1\pa{Y,y_1}$ is not surjective.
Let $q$ be the object $\pr{1}:C\times F\to X$ (this is a covering map by Lemma~\ref{cover_component}).
Components of $\fa\pa{E}$ are in bijective correspondence with $F$ by Lemma~\ref{comp_bp} and~\cite[Prop.~1.31]{hatcher}.
If $Z$ is a component of $\fa\pa{E}$, then the based object $\pa{Z,\pa{x_1,e_i}}\to \pa{C,x_1}$ is based isomorphic to the trivial object $\tn{id}:\pa{C,x_1}\to\pa{C,x_1}$ by Lemma~\ref{key}.
It follows that $\fa\pa{p}\cong q$ (unbased), say by $\psi:\fa\pa{E}\to C\times F$.
We have the morphism $\sigma:C\times F\to C\times F$ given by $\pa{x,e_i}\mapsto\pa{x,e_1}$.
Thus, $s:=\psi^{-1}\circ\sigma\circ\psi$ is a morphism in $\h{\fa\pa{p},\fa\pa{p}}$.
As $\fa$ is full, there exists $t\in\h{p,p}$ such that $\fa\pa{t}=s$.
As $\sigma\pa{x_1,e_1}=\pa{x_1,e_1}$, we get $t\pa{e_1}=e_1$.
As $E$ is connected, Corollary~\ref{equalizer_morphisms} implies $t$ is the identity.
Thus, $\fa\pa{t}$ is the identity, a contradiction (since $\card{F}\geq2$).
Hence, $\fs:\pi_1\pa{X,x}\to\pi_1\pa{Y,y}$ is surjective for each $x\in X$.\\

The argument in the previous paragraph adapts readily to the based and surjective categories.
For all three categories, consider the (extrinsic) disjoint union of $p$ and a trivial one-sheeted cover of $Y$.
Base at the unique point above $y_0$ in the added trivial cover.
Now, the same argument applies. This completes the proof of the backward implication.\\

For the forward implication, Corollary~\ref{gq1_pc} implies that $\fa$ is faithful.
To show $\fa$ is full, let $p_1:E_1\to Y$ and $p_2:E_2\to Y$ be objects.
Let $s\in\h{\fa\pa{E_1},\fa\pa{E_2}}$.
We have the commutative diagram:
\begin{equation}\begin{split}\label{seek_t}
\xymatrix{
    \fa\pa{E_{1}}	\ar[rrr]	\ar[rd]^-{\fa\pa{p_1}}	\ar[dd]_{s}	&	&	&	E_{1}  \ar[dl]_{p_{1}} \ar@{-->}[dd]^{t}\\
    &	X	\ar[r]^{f}																																								&	Y\\
    \fa\pa{E_{2}}	\ar[rrr]	\ar[ru]_-{\fa\pa{p_2}}															&	&	&	E_{2}  \ar[ul]^{p_{2}}}
\end{split}\end{equation}
We seek $t\in\h{E_1,E_2}$ such that $\fa\pa{t}=s$.
It suffices to specify $t$ on each component of $E_1$.
So, fix a component $C_1$ of $E_1$.
By Lemma~\ref{imageclosed}, $p_1\pa{C_1}$ is a component of $Y$.
As $\fs:\pi_0\pa{X}\to\pi_0\pa{Y}$ is surjective, there exists $x_1\in X$ such that $y_1:=f\pa{x_1}\in p_1\pa{C_1}$.
Let $c_1\in p_1^{-1}\pa{y_1}\cap C_1$.
Thus, $\pa{x_1,c_1}\in\fa\pa{C_1}$.
Lemma~\ref{inverseopenclosed} implies $\fa\pa{C_1}$ is open and closed in $\fa\pa{E_1}$, and Lemma~\ref{pb_connected} implies $\fa\pa{C_1}$ is connected.
Now, $s\pa{x_1,c_1}=\pa{x_1,c_2}$ for some $c_2\in E_2$, and so $p_2\pa{c_2}=y_1$.
Let $C_2$ be the component of $E_2$ containing $c_2$.
By Lemma~\ref{imageclosed}, $p_2\pa{C_2}$ is a component of $Y$.
As $y_1$ lies in $p_1\pa{C_1}$ and $p_2\pa{C_2}$, we get $p_1\pa{C_1}=p_2\pa{C_2}$.
Two applications of Lemma~\ref{key} and commutativity of the triangle in~\eqref{seek_t} imply that:
\begin{equation}\label{temp_f_sharp}
	f_\sharp^{-1}\pa{\pa{p_1}_\sharp\pa{\pi_1\pa{C_1,c_1}}}\subset f_\sharp^{-1}\pa{\pa{p_2}_\sharp\pa{\pi_1\pa{C_2,c_2}}}.
\end{equation}
As $\fs:\pi_1\pa{X,x_1}\to\pi_1\pa{Y,y_1}$ is surjective, equation~\eqref{temp_f_sharp} implies that:
\[
	\pa{p_1}_\sharp\pa{\pi_1\pa{C_1,c_1}}\subset \pa{p_2}_\sharp\pa{\pi_1\pa{C_2,c_2}}.
\]
Therefore, there is a unique lift $\tau:\pa{C_1,c_1}\to\pa{C_2,c_2}$ of $\rest{p_1}C_1$ to $C_2$.
In particular, $\rest{p_1}C_1=\pa{\rest{p_2}C_2}\circ\tau$.
Lemma~\ref{morphismiscover} implies that $\tau$ is a covering map.
Lemma~\ref{imageclosed} implies that $\tau$ is surjective.
Lemma~\ref{pullbacktriangle} implies that $\fa\pa{\tau}:\fa\pa{C_1}\to\fa\pa{C_2}$ is surjective.
Evidently, $\fa\pa{\tau}$ and $\rest{s}\fa\pa{C_1}$ agree at $\pa{x_1,c_1}$.
As $\fa\pa{C_1}$ is connected, Corollary~\ref{equalizer_morphisms} implies that $\fa\pa{\tau}=\rest{s}\fa\pa{C_1}$.
Hence, the following diagram commutes.
\begin{equation}\begin{split}\label{tau_diag}
\xymatrix{
    \fa\pa{C_{1}}	\ar[rrr]	\ar[rd]^-{\rest{\fa\pa{p_1}}}	\ar[dd]_{\fa\pa{\tau}=\rest{s}}	&	&	&	C_{1}  \ar[dl]_{\rest{p_{1}}} \ar[dd]^{\tau}\\
    &	X	\ar[r]^{f}																																								&	Y\\
    \fa\pa{C_{2}}	\ar[rrr]	\ar[ru]_-{\rest{\fa\pa{p_2}}}															&	&	&	C_{2}  \ar[ul]^{\rest{p_{2}}}}
\end{split}\end{equation}
Define $\rest{t}C_1:=\tau$.
Evidently, diagram~\eqref{seek_t}, with $t$ included, commutes, and $\fa\pa{t}=s$.
If $s$ is surjective, then $t$ is surjective.
If the data are based, then $t$ respects basepoints.
Thus, $t\in\h{E_1,E_2}$.
This completes the proof of Proposition~\ref{gq2}.
\end{proof}

\subsection{Generalizing Quillen's~\ref{q3}}\label{ss:gq3}

The following is our generalization of~\ref{q3}.

\begin{proposition}\label{gq3}
Fix a category of coverings.
Let $f:X\to Y$ be a map where $X$ and $Y$ are locally path-connected and semilocally simply-connected.
Then, \hbox{$\fs:\pi_0\pa{X}\to\pi_0\pa{Y}$} is a bijection and $\fs:\pi_1\pa{X,x}\to\pi_1\pa{Y,f\pa{x}}$ is an isomorphism for each $x\in X$ if and only if $\fa$ is an equivalence of categories.
\end{proposition}

\begin{proof}[Proof of Proposition~\ref{gq3}]
Recall the well known characterization: a functor is an equivalence of categories if and only it is full, faithful, and essentially surjective~\cite[p.~93]{maclane}.\\

We begin with the backwards implication. By Proposition~\ref{gq2}, it remains to prove $\fs:\pi_1\pa{X,x}\to\pi_1\pa{Y,f\pa{x}}$ is injective for each $x\in X$.
Suppose, by way of contradiction, that $\fs:\pi_1\pa{X,x_1}\to\pi_1\pa{Y,y_1}$ is not injective where $y_1:=f\pa{x_1}$.
Let $C$ be the component of $X$ containing $x_1$.
Let $q:\pa{Z,z_1}\to\pa{C,x_1}$ be a connected and simply-connected covering (here we use semilocal simple-connectedness of $X$).
By Lemma~\ref{cover_component}, $p:\pa{Z,z_1}\to\pa{X,x_1}$ is a covering map where $p(z):=q(z)$.
By Lemma~\ref{key}, there is no object over $Y$ whose pullback is isomorphic to $p$.
But, this contradicts the hypothesis that $\fa$ is essentially surjective.
For the surjective and based categories, consider the (extrinsic) disjoint union of $p$ and a trivial one-sheeted cover of $Y$.
Base at the unique point above $x_0$ in the trivial cover.
Now, the same argument applies.
This completes the proof of the backward implication.\\

Next, we prove the forward implication. By Proposition~\ref{gq2}, it remains to prove $\fa$ is essentially surjective.
Let $p:Z\to X$ be an object.
Then, $Z=\bigsqcup_{i\in I} Z_i$ is an intrinsic disjoint of its components (below, it is more convenient to index by $i\in I$ rather than by $Z_i\in\pi_0\pa{Z}$).
Fix a component $Z_i$, $i\in I$, of $Z$.
By Lemma~\ref{imageclosed}, $X_i:=p\pa{Z_i}$ is a component of $X$.
Let $z_i \in Z_i$ and define $x_i:=p\pa{z_i}$.
By Corollary~\ref{decompose_cover_componentwise}, the restriction $p_i: \pa{Z_i,z_i} \to \pa{X_i,x_i}$ of $p$ is a (based and surjective) covering map,
and the restriction $P_i:\pa{Z_i,z_i} \to \pa{X,x_i}$ of $p$ is a (based) covering map.
Define $y_i:=f\pa{x_i}$.
To the subgroup:
\[
 \fs\pa{  \pa{p_i}_\sharp \pa{\pi_1\pa{Z_i,z_i}}} \subset \pi_1\pa{Y_i,y_i}
\]
there corresponds a connected cover $q_i:\pa{E_i,e_i}\to\pa{Y_i,y_i}$.
By Lemma~\ref{cover_component}, $Q_i:\pa{E_i,e_i}\to\pa{Y,y_i}$ is a covering map where $Q_i\pa{e}:=q_i\pa{e}$.
By hypothesis, $\fs:\pi_1\pa{X,x_i}\to\pi_1\pa{Y,y_i}$ is an isomorphism.
Lemma~\ref{pb_connected} implies that $\fa\pa{E_i}$ is connected.
Lemmas~\ref{key} and~\ref{basic_group_iso} imply that $\fa\pa{Q_i}$ and $P_i$ are isomorphic objects in $\bcov\pa{X,x_i}$.
All disjoint unions are over $i\in I$.
Define $E:=\coprod E_i$ and $Q:=\coprod Q_i$.
By Lemma~\ref{union_covers}, $Q:E\to Y$ is a covering map.
For the category $\cov$, we have:
\begin{equation}\label{string_isos}
	\fa\pa{Q}= \fa\pa{\coprod Q_i}	\cong \coprod \fa\pa{Q_i} \cong \coprod P_i \cong p
\end{equation}
where the first and second isomorphisms follow by Lemmas~\ref{pbe} and~\ref{isodu} respectively, and the last isomorphism is trivial since $Z=\bigsqcup Z_i$.
If $p$ is surjective, then $Q$ is surjective.
If the data $f:\bs{X}\to \bs{Y}$ and $p:\bs{Z}\to \bs{X}$ are based, then let $Z_0$ denote the component of $Z$ containing $z_0$ and, naturally, base $E$ at $\pa{e_0,0}$.
Thus, \eqref{string_isos} holds in all four categories of coverings.
The proof of Proposition~\ref{gq3} is complete.
\end{proof}

The previous proof used semilocal simple-connectedness of $X$ only in the backward implication to deduce that $\fs:\pi_1\pa{X,x}\to\pi_1\pa{Y,f\pa{x}}$ is injective for each $x\in X$.
This implication does not hold in general when $X$ is not semilocally simply-connected, as shown by the following example.

\begin{example}\label{ha_ex}
Let $X$ denote the \emph{Harmonic archipelago}, an interesting noncompact subspace of $\R^3$ discovered by Bogley and Sieradski~\cite[pp.~6--7]{bogley_sieradski} and defined as follows.
Consider the unit disk $D$ in $\R^2\times\cpa{0}$ containing a nice copy of the Hawaiian earring (see Example~\ref{he_ex} above) with wild point $x_0:=\pa{-1,0,0}$.
For each pair of successive circles $C_n$ and $C_{n+1}$ in the Hawaiian earring, let $D_n$ be a nice round subdisk of $D$ between $C_n$ and $C_{n+1}$ and centered on the $x$-axis.
\begin{figure}[h!]
    \centerline{\includegraphics[scale=0.70]{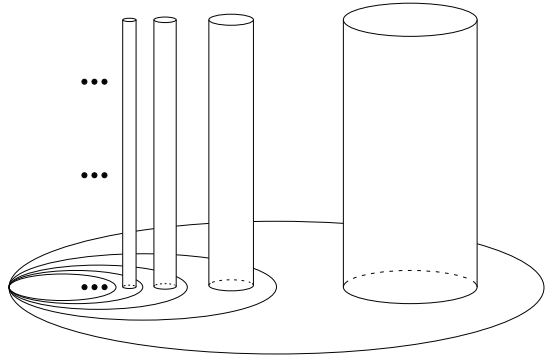}}
    \caption{Harmonic archipelago  $X\subset\R^3$.}
    \label{harm_arch}
\end{figure}
Replace $D_n$ with a parallel copy of $D_n$, raised up a fixed height $h>0$, and include the vertical annulus stretching between their boundaries (see~Figure~\ref{harm_arch}).
The Harmonic archipelago $\bs{X}$ is the resulting based space.
It is not difficult to verify that $X$ is not semilocally simply-connected at $x_0$, and, nonetheless, all coverings of $X$ are trivial.
Let $Y:=\cpa{y_0}$ be a point.
Let $f:\bs{X}\to\bs{Y}$ be the constant map.
Thus, $\fa$ is an equivalence of categories, although $\fs:\pi_1\bs{X}\to\pi_1\bs{Y}$ is not injective.
\end{example}


\begin{thebibliography}{9}

\bibitem[BN02]{barnatan}
	D.~Bar-Natan,
	\emph{Covering spaces done right},
	notes, available at \href{http://www.math.toronto.edu/drorbn/classes/0102/AlgTop/CoveringsDoneRight.pdf}{\curl{http://www.math.toronto.edu/drorbn/classes/0102/AlgTop/CoveringsDoneRight.pdf}} (2002), 2 pp.


\bibitem[BS98]{bogley_sieradski}
	W.A.~Bogley and A.J.~Sieradski,
	\emph{Universal path spaces},
	preprint, available at \href{http://people.oregonstate.edu/~bogleyw/research/ups.pdf}{\curl{http://people.oregonstate.edu/\textasciitilde bogleyw/research/ups.pdf}} (1998), 50 pp.

\bibitem[GZ67]{gabriel_zisman}
	P.~Gabriel and M.~Zisman,
	\emph{Calculus of fractions and homotopy theory},
	Springer-Verlag, New York, 1967.

\bibitem[Har77]{hartshorne}
	R.~Hartshorne,
	\emph{Algebraic geometry},
	Springer-Verlag, New York, 1977.

\bibitem[Hat02]{hatcher}
	A.~Hatcher,
	\emph{Algebraic topology},
	Cambridge University Press, Cambridge, 2002.

\bibitem[Hir94]{hirsch}
	M.W.~Hirsch,
	\emph{Differential topology},
	corrected reprint of the 1976 original,
  Springer-Verlag, New York, 1994.

\bibitem[JS91]{joyal_street}
	A.~Joyal and R.~Street,
	\emph{An introduction to Tannaka duality and quantum groups},
	in Lecture Notes in Math. \textbf{1488},
	Springer, Berlin, 1991, 413--492.

\bibitem[Mac98]{maclane}
	S.~Mac Lane,
	\emph{Categories for the working mathematician},
	Second edition,
	Springer-Verlag, New York, 1998.


\bibitem[M{\o}l11]{moller}
	J.~M{\o}ller,
	\emph{The fundamental group and covering spaces},
	notes, available at \href{http://arxiv.org/abs/1106.5650}{\curl{http://arxiv.org/abs/1106.5650}} (2011), 31 pp.

\bibitem[Qui78]{quillen}
	D.~Quillen,
	\emph{Homotopy properties of the poset of nontrivial {$p$}-subgroups of a group},
	Adv. in Math. \textbf{28} (1978), 101--128.

\bibitem[Spa81]{spanier}
	E.H.~Spanier,
	\emph{Algebraic topology},
	corrected reprint, Springer-Verlag, New York, 1981.


\bibitem[Ste99]{steenrod}
	N.~Steenrod,
	\emph{The topology of fibre bundles},
	reprint of the 1957 ed.,
	Princeton University Press,
	Princeton, NJ, 1999.

\bibitem[Sza09]{szamuely}
	T.~Szamuely,
	\emph{Galois groups and fundamental groups},
	Cambridge University Press, Cambridge, 2009.

\end{thebibliography}
\end{document}